\documentclass[12pt]{article}

\usepackage{hyperref}
\usepackage{amsmath}
\usepackage{amsfonts}
\usepackage{amssymb}
\usepackage{mathtools}
\usepackage{graphicx}
\usepackage{natbib}
\usepackage{xcolor}
\usepackage{xr}
\usepackage{authblk}
\newtheorem{theorem}{Theorem}
\newtheorem{corollary}[theorem]{Corollary}

\newtheorem{definition}[theorem]{Definition}

\newtheorem{remark}[theorem]{Remark}
\newtheorem{lemma}[theorem]{Lemma}
\newtheorem{proposition}[theorem]{Proposition}

\newenvironment{proof}[1][Proof]{\noindent\textbf{Proof of #1.} }{\ \rule{0.5em}{0.5em}}

\newcommand{\norm}[1]{\left\lVert#1\right\rVert}
\newcommand{\abs}[1]{\left\lvert#1\right\rvert}
\DeclareMathOperator{\med}{Med}

\DeclarePairedDelimiterX{\inp}[2]{\langle}{\rangle}{#1, #2}
\newcommand{\eqd}{\ensuremath{\overset{d}{=}}}
\renewcommand{\d}[1]{\ensuremath{\operatorname{d}\!{#1}}}

\begin{document}

\title{Generalization of the simplicial depth: no vanishment outside the convex hull of the distribution support}

\author{Giacomo Francisci, Alicia Nieto-Reyes and Claudio Agostinelli}

\newcommand{\addresses}{{
  \bigskip
  \footnotesize

  G.~Francisci, \textsc{Department of Statistics, George Mason University,
    Fairfax 22030, Virginia, U.S.A.} \textit{E-mail address:} \texttt{gfranci@gmu.edu}

  \medskip

  A.~Nieto-Reyes, \textsc{Department of Mathematics, Statistics and Computation, University of Cantabria, Santander 39005, Spain.} \textit{E-mail address:} \texttt{alicia.nieto@unican.es}

  \medskip

  C.~Agostinelli, \textsc{Department of Mathematics, University of Trento, Trento 38122, Italy.} \textit{E-mail address:} \texttt{claudio.agostinelli@unitn.it}

}}

\date{\today}

\maketitle

\begin{abstract} The simplicial depth, like other relevant multivariate statistical data depth functions, vanishes right outside the convex hull of the support of the distribution with respect to which the depth is computed.  This is problematic when it is required to differentiate among points outside the convex hull of the distribution support, with respect to which the depth is computed, based on their depth values. We provide the first proposal for simplicial depth which do not vanish right outside the convex hull of the distribution. The  properties of the proposal and of the corresponding estimator are studied theoretically and by means of Monte Carlo simulations and analysis of datasets. \\

\noindent \textbf{Keywords}: Classification, consistency, empirical depth, multivariate statistical data depth, multivariate symmetry, vanishment outside the convex hull.
\end{abstract}

\section{Introduction}
\label{section:introduction}

Multivariate statistical data depth functions provide an order of the elements of a space on $\mathbb{R}^p,$ $p\geq 1,$ with respect to a probability distribution on the space. 
\citet{Liu-1990} introduced the simplicial depth as an instance of depth that satisfies some good theoretical properties that later became the constituting properties of the notion of statistical data depth \citep{Zuo-2000}. The simplicial depth of a $p$-dimensional point $x$ with respect to a distribution $P$ is the probability that $x$ is contained in a closed simplex whose vertices are drawn independently from $P$, namely
\begin{equation}\label{S}
d_S(x; P) = \int \mathbf{I}( x \in \triangle[x_{1},\dots,x_{p+1}]) \d P(x_1) \cdots \d P(x_{p+1});
\end{equation}
 with $\triangle[x_1, \dots, x_{p+1}]$ denoting the closed simplex with vertices on $x_1, \dots, x_{p+1}$ and $\mathbf{I}$  the indicator function. Given a random sample $X_1, \dots, X_n$ from $P$, the simplicial depth is estimated by the sample simplicial depth
 \begin{equation}\label{ss}
 d_{S,n}(x; P) = \binom{n}{p+1}^{-1} \sum_{1 \leq i_1 < \dots < i_{p+1} \leq n} \mathbf{I} ( x \in \triangle [ X_{i_1}, \dots, X_{i_{p+1}} ] ).
 \end{equation}
 The simplicial depth has been broadly studied and applied in the literature \citep[for instance]{Arcones-1994,Li-2012}, with consistency of the sample simplicial depth being established in \citet{Liu-1990} and \citet{Arcones-1993}. 

 Other well-known depth functions are Tukey \citep{Tukey-1975} and spatial depth \citep{Vardi-2000, Serfling-2002}. \citet{Girard-2017} proves that the spatial depth is always positive, while this is not the case for neither simplicial  nor Tukey depth. 
 A particularity of simplicial and Tukey depth functions is that they provide zero depth value to the elements of the space that are outside the convex hull of the distribution support with respect to which the depth is computed. 
 Additionally, sample Tukey and sample simplicial depth functions give value zero to every point in the space outside the convex hull of the observed sample. These points are called outsiders \citep{Lange-2014}.  
This is problematic for those applications in which it is required to discriminate among different elements of the space (see the following paragraphs).  
   An approach to tackle this issue is pursued in \citet{Einmahl-2015} for the sample version of the Tukey depth. A more recent attempt to improve the Tukey depth is presented in \citet{Nagy-2021}.  Another existing depth function is zonoid depth \citep{Koshevoy-1997}. For the sample zonoid, an approach is presented in \citet{Mosler-2006}, where, for the outsiders, it is suggested to use the Mahalanobis depth which is positive everywhere. Alternatively, the projection depth \citep{Zuo-2003} and the spatial depth can be used. Additionally, \citet{Lange-2014} suggests to use $k$-nearest neighbors to classify the outsiders. The literature contains no such approaches for the simplicial depth.

Given two samples $X=\{X_1, \ldots, X_m\}$ and $Z=\{Z_1, \ldots, Z_o\},$ drawn respectively from distributions $P^{(1)}$ and $P^{(3)}$ and a third sample $Y=\{Y_1, \ldots, Y_n\}$ drawn from an unknown distribution, say $P^{(2)}$, consider the problem of classifying each element of $Y$ according to its similarity to observations from $P^{(1)}$ or $P^{(3)}.$ 
As an example, we make use of the Alzheimer dataset,  which
was first introduced in \citet{Nieto-Reyes-2017} and later used in \citet{Bringas-2020} and in \citet{Nieto-Reyes-2021-b}. 
 This dataset  consists of recording accelerations for 35 patients at different stages of the Alzheimer's disease: 7 in a mild stage, 18 in a moderate stage and 10 in a severe stage of the disease. There are various acceleration recordings per patient, which vary between 2 and 8 depending on the patient. Each recording is registered at a different grid and for a different amount of time. 

In  \citet{Nieto-Reyes-2021-b} two measurements are used to summarize the acceleration over time: (i) the distance from the center of symmetry \citep{Nieto-Reyes-2021-a} and (ii) the median absolute deviation of the distances restricted to the corresponding stage of the disease. This is illustrated in Figure \ref{figure:alzheimer:dataset}, 
\begin{figure}[!htb]
\begin{center}
\includegraphics[width=0.99\textwidth]{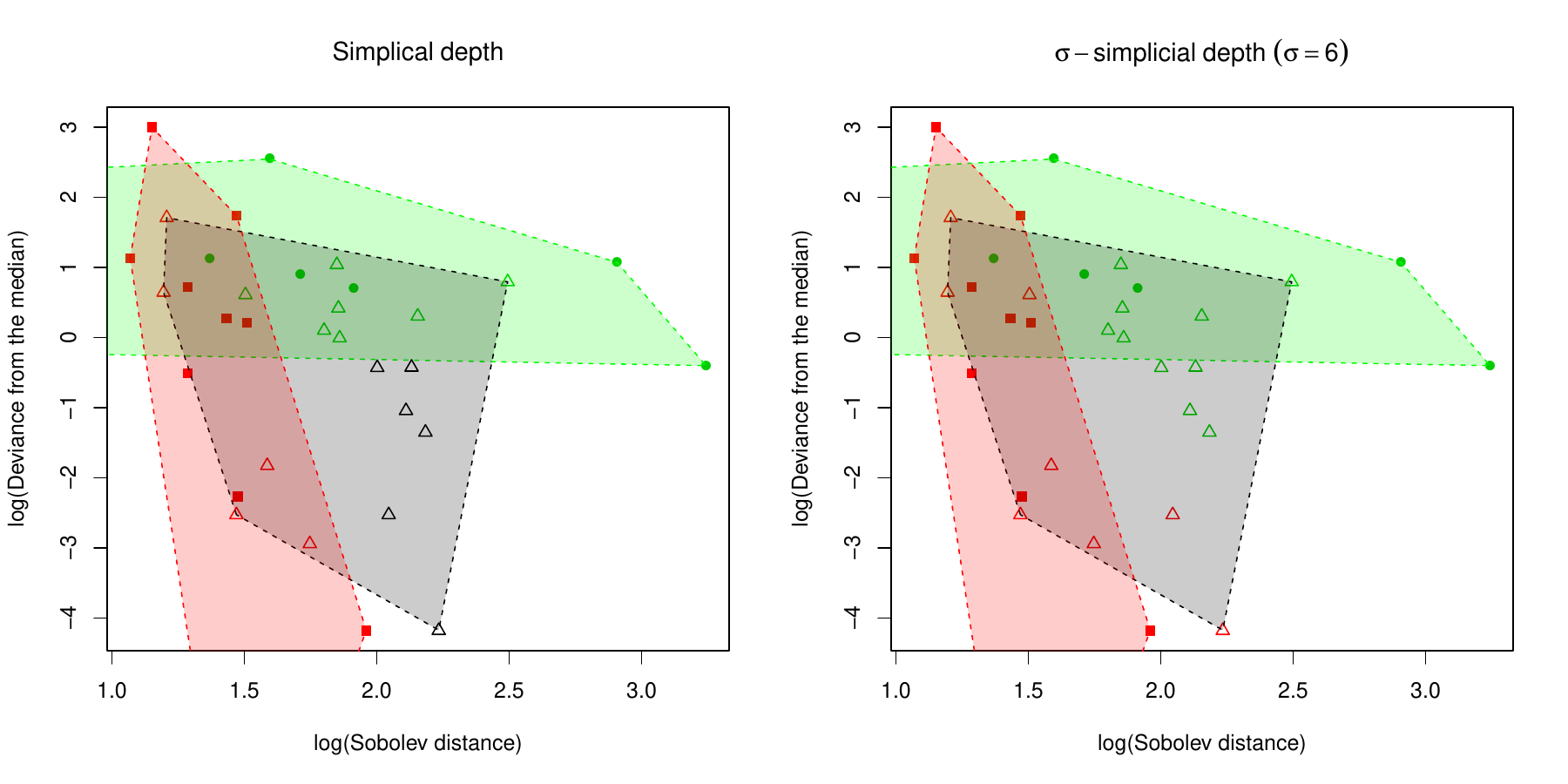}
\end{center}
\caption{Alzheimer dataset. Group 1 (mild stage): filled green circle , group 2 (moderate stage): triangle, group 3 (severe stage): filled red square. Left plot: simplicial depth, right plot: simplex enlarged $\sigma$-simplicial depth with $\sigma=6$. Units in group 2 are classified in group 1 (green) or group 3 (red). Black triangles are unclassified units with zero simplicial depth.}
\label{figure:alzheimer:dataset}
\end{figure}
where the $X$ (mild stage) and $Z$ (severe stage) are plotted respectively as filled green circles and filled red squares, while $Y$ (moderate stage) observations are represented by triangles. Notice that there are elements of $Y$ which are outside both convex hulls of $X$ and $Z.$  If we were to classify the elements of $Y$ using a supervised classification methodology based on statistical depth \citep[see for example,][]{Li-2012}, we would like the depth value of each $Y_i,$ $i=1, \ldots, n,$ with respect to the empirical distribution associated to $X,$ to differ from the value obtained when computed with respect to the empirical distribution associated to $Z.$
However, the sample simplicial of points outside both convex hulls is zero. 

We propose a generalization of simplicial depth which allows to classify all the points, see right panel of Figure \ref{figure:alzheimer:dataset}. Among the $6$ outsiders $4$ are classified as in mild stage and $2$ as in severe stage.
This methodology, that generalizes the simplicial depth, results in depth functions that do not vanish right outside the convex hull of the distribution support. The methodology is based on the idea of enlargement, which is performed in two different ways. The first approach is an enlargement of the simplex involved in the definition of the simplicial depth and the second approach might be seen as an enlargement of the distribution with respect to which the depth is computed; which involves linear combinations of independent random variables. The  enlargement of the simplex is the approach used in the right panel of Figure \ref{figure:alzheimer:dataset}, where $\sigma$ is a parameter controlling the enlargement size.

As multivariate symmetry \citep{Zuo-2000-b} is a key concept in the notion of depth, we dedicate Section \ref{section:symmetry:of:random:variables} to study the conditions under which symmetry is inherited under affine combinations. This has far reaching implications as affine combinations are widely applied in statistics, for instance, in dimension reduction problems. Moreover, the non-symmetry of an affine combination will imply the non-symmetry of the original distribution, under certain assumptions. Section \ref{Simplicial} introduces the two different generalizations of the simplicial depth. The results concerning multivariate symmetry are exploited to derive useful properties for the simplicial depth based on an enlargement of the distribution. We provide sample versions that do not vanish right outside the convex hull of the sample and study their consistency. Section \ref{Simulations} presents some Monte Carlo simulations where the two main contributions are studied empirically. Section \ref{data_analysis} illustrates the performance of the introduced methods on four datasets. The proofs of the results stated in the following sections are provided in the supplementary material.
 
\section{Symmetry of random variables}
\label{section:symmetry:of:random:variables}

We prove under which notions of symmetry the affine combinations of independent and symmetric random variables are symmetric. The most well-known notions of symmetry in the literature are spherical, elliptical, central, angular and halfspace symmetry, where each is a generalization of the previous one \citep{Serfling-2004}. A random variable $X \in \mathbb{R}^p$ is \emph{spherically symmetric} about a point $\mu \in \mathbb{R}^p$ if $X-\mu$ and $U(X-\mu)$ are identically distributed for any orthonormal matrix $U$. A random variable $X$ in $\mathbb{R}^p$ is \emph{elliptically symmetric} about a point $\mu \in \mathbb{R}^p$ if there exists a nonsingular matrix $V$ such that $VX$ is spherically symmetric about $V \mu$ \citep{Ley-2011}. A random variable $X$ in $\mathbb{R}^p$ is \emph{centrally symmetric} about a point $\mu \in \mathbb{R}^p$ if $X-\mu$ and $\mu-X$ are identically distributed. Note that the notions of spherical, elliptical and central symmetry coincide for univariate random variables.
The notion of angular symmetry was introduced in \citet{Liu-1990}: a random variable $X$ in $\mathbb{R}^p$ is \emph{angularly symmetric} about a point $\mu \in \mathbb{R}^p$ if $(X-\mu)/\norm{X-\mu}$ and $(\mu-X)/\norm{X-\mu}$ are identically distributed. 
This was generalized in \citet{Zuo-2000-b} by defining a random variable $X$ in $\mathbb{R}^p$ to be \emph{halfspace symmetric} about $\mu$ if $\mathbb{P}(X \in H) \geq \frac{1}{2}$ for every closed halfspace $H$ with $\mu$ on the boundary. 

Next proposition states that these notions of symmetry are preserved under translation and scalar multiplication.
\begin{proposition} \label{proposition:enlargement:translation}
	Let $X$ be a random variable on $\mathbb{R}^p$ that is symmetric about $\mu \in \mathbb{R}^p$ with respect to either spherical, elliptical, central, angular or halfspace symmetry. 
	Then, for any $\lambda \in \mathbb{R}$ and $b \in \mathbb{R}^p,$ $\lambda X + b$ is symmetric about $\lambda \mu + b$ with respect to the same notion of symmetry.
\end{proposition} 
If a distribution is spherically, elliptically or centrally symmetric, the center of symmetry is unique. If the distribution is angular or halfspace symmetric, the center is unique but for the degenerate case in which the distribution on $\mathbb{R}^p,$ $p>1,$ has all its probability mass on a line with more than one median \citep[Theorem 2.1, Lemma 2.3]{Zuo-2000-b}.
Note that, when the center of symmetry is not unique, Proposition \ref{proposition:enlargement:translation} remains valid for each center of symmetry.  

Next two results concern the inheritance under affine combinations of spherical, elliptical and central symmetry and, in general, they do not hold for angular and halfspace symmetry.
\begin{proposition} \label{proposition:elliptical:symmetry}
Let $X_1, \ldots, X_n$ be independent and identically distributed random variables on $\mathbb{R}^p$ that are symmetric about $\mu \in \mathbb{R}^p$ with respect to either spherical, elliptical or central symmetry. 
For any $\lambda_1, \ldots, \lambda_n \in \mathbb{R}$ and $b \in \mathbb{R}^p$ then $\sum_{i=1}^{n} \lambda_i X_i + b$ is symmetric about $\sum_{i=1}^{n} \lambda_i \mu + b$ with respect to the same notion of symmetry.
\end{proposition}
The above proposition is generalized below to non-identically distributed random variables for the notions of spherical and central symmetry.
\begin{proposition} \label{proposition:spherical:central:symmetry}
Let $X_1, \ldots, X_n$ be independent random variables on $\mathbb{R}^p.$ For either spherical or central symmetry,  let $X_i$ be symmetric about $\mu_i \in \mathbb{R}^p$ for all $i=1, \ldots, n.$
For any $\lambda_1, \ldots, \lambda_n \in \mathbb{R}$ and $b \in \mathbb{R}^p$ then $\sum_{i=1}^{n} \lambda_i X_i + b$ is symmetric about $\sum_{i=1}^{n} \lambda_i \mu_i + b$ with respect to the same notion of symmetry.
\end{proposition}
The result in Proposition \ref{proposition:spherical:central:symmetry} does not hold for elliptically symmetric distributions. To see this consider two multivariate random variables $X_1$, $X_2$, which are elliptically symmetric with means $\mu_1$, $\mu_2$ and covariance matrices $\Sigma_1$, $\Sigma_2$. By \citet{Frahm-2004}[Proposition 1, Theorem 2], they have characteristic functions $\varphi_{X_1}(x)=\exp(i x^{\top} \mu_1) \, \phi_1(x^{\top} \Sigma_1 x)$, $\varphi_{X_2}(x)=\exp(i x^{\top} \mu_2) \, \phi_2(x^{\top} \Sigma_2 x)$ for some functions $\phi_1,\phi_2$. Now, the sum $X_1+X_2$ has characteristic function $\varphi_{X_1+X_2}(x) = \varphi_{X_1}(x) \, \varphi_{X_2}(x)$, and it does not fall back in the previous form unless $\phi_1(s) \, \phi_2(t)=\phi(s+t)$, for some function $\phi$. This is true for normal distributions, in which case $\phi_1$, $\phi_2$ have exponential form, but it is not true in general. 

For elliptically symmetric non-identically distributed random variables, we have the following corollary of Proposition \ref{proposition:spherical:central:symmetry}. This is due to the facts that Proposition \ref{proposition:spherical:central:symmetry} holds for central symmetry and that central symmetry  is a generalization of elliptical symmetry.

\begin{corollary}\label{C4}
Let $X_1, \ldots, X_n$ be as in Proposition \ref{proposition:spherical:central:symmetry}, but with $X_i$ elliptically symmetric about $\mu_i \in \mathbb{R}^p$ for each $i=1, \ldots, n.$ Then,  $\sum_{i=1}^{n} \lambda_i X_i + b$ is centrally symmetric about $\sum_{i=1}^{n} \lambda_i \mu_i + b$ for any $\lambda_1, \ldots, \lambda_n \in \mathbb{R}$ and $b \in \mathbb{R}^p$.
\end{corollary}
 
\section{Generalization of the simplicial depth}
\label{Simplicial}

The objective of this section is to modify the simplicial depth, in \eqref{S},  in a manner that  the points in $\mathbb{R}^p$ that are outside the convex hull of the  support of $P$ do not necessarily have depth value zero, when the depth is computed with respect to $P.$ We pursue it in two different manners: Definition \ref{sigma:simplicial:depth:dependent:vertices} uses an enlargement of the simplex and Definition \ref{sigma:simplicial:depth:independent:vertices} computes the simplicial depth with respect to a transformation of the original distribution with respect to which the depth is evaluated.

\begin{definition}
\label{sigma:simplicial:depth:dependent:vertices}
Given $\sigma> 0$, the simplex enlarged \emph{$\sigma$-simplicial depth} of a point $x \in \mathbb{R}^p$ with respect to a distribution $P$ on $\mathbb{R}^p$ is 
\begin{equation*}
d_{\triangle_\sigma}(x; P) := \int \mathbf{I}( x \in \triangle_\sigma[x_{1},\ldots,x_{p+1}]) \d P(x_1) \cdots \d P(x_{p+1}),
\end{equation*}
where $\triangle_\sigma[x_{1},\ldots,x_{p+1}] := \triangle[y_1,\ldots,y_{p+1}]$ with $y_i = \sigma (x_i - \bar{x}) + \bar{x}$ and $\bar{x} := \sum_{j=1}^{p+1} x_j/(p+1).$

\end{definition}
For $\sigma=1$ there is no enlargement of the simplex and we are left with the simplicial depth. The case $0 < \sigma < 1$ corresponds to a reduction of the simplex, whereas for $\sigma=0$ the simplex degenerates into its centroid $\bar{x}$. 

\begin{definition}\label{Psigma}
Given $P$ a distribution on $\mathbb{R}^p$ and $\sigma > 0,$ the \emph{$\sigma$-transformation} of $P,$ $P_{\sigma},$ is the distribution of the random variable $\sigma (X_1 - \bar{X}) + \bar{X},$  where $X_1, \ldots, X_{p+1}$ are independent and identically distributed random variables with distribution $P$ and $\bar{X} := \sum_{j=1}^{p+1} X_j/(p+1).$ 
\end{definition}
We turn to the definition of the simplicial depth with respect to a transformation of the original distribution.
\begin{definition}\label{sigma:simplicial:depth:independent:vertices}
Given $\sigma> 0,$ the distribution enlarged \emph{$\sigma$-simplicial depth} of a point $x \in \mathbb{R}^p$ with respect to a distribution $P$ on $\mathbb{R}^p$ is 
$$d_{P_\sigma}(x; P):= d_S(x; P_\sigma) = \int \mathbf{I}( x\in\triangle[x_{1},\dots,x_{p+1}]) \d P_\sigma(x_1) \cdots \d P_\sigma(x_{p+1}).$$
\end{definition}
We note that the simplex enlarged $\sigma$-simplicial depth is a refinement of the simplicial depth that makes use of dependent vertices in the simplex while the distribution enlarged $\sigma$-simplicial depth benefits from independent vertices in the simplex.
It is also worth highlighting that the proposed $\sigma$-simplicial depths can have depth value zero outside a region that depends on the selected $\sigma.$ However, this does not pose any problems as it is always possible to choose $\sigma$ such that a region of interest has positive depth value.


Before studying the theoretical properties of the two types of $\sigma$-simplicial depth, we examine the inheritance of regularity conditions from $P$ to $P_{\sigma}$. For this we make use of some of the results in Section \ref{section:symmetry:of:random:variables}. These conditions will be useful in proving most of the properties of the distribution enlarged $\sigma$-simplicial depth. We recall that a distribution $P$ is \emph{smooth}  if 	$P(L) = 0$ for any hyperplane $L \subset \mathbb{R}^p$ \citep{Masse-2004}.

\begin{proposition} \label{PSigmaCont}
Let $P$ be a distribution on $\mathbb{R}^p$ and $\sigma > 0.$ If $P$ is a continuous distribution (or respectively absolutely continuous or smooth), the $\sigma$-transformation of $P,$ $P_{\sigma},$ is a continuous distribution (or respectively absolutely continuous or smooth).
\end{proposition}
The following result concerns the inheritance of symmetry that  $P_{\sigma}$ acquires  from $P.$
\begin{proposition} \label{closure:F:sigma:under:symmetry}
Let $P$ be a distribution on $\mathbb{R}^p$ and $\sigma > 0.$ If $P$ is either spherical, elliptical or centrally symmetric about $\mu\in\mathbb{R}^p,$ then $P_{\sigma}$ is symmetric about $\mu$ with respect to the same notion of symmetry.
 \end{proposition}
 \begin{remark}
If the mean of a spherically, elliptically or centrally symmetric distribution exists, then it coincides with the center of symmetry and, by Proposition \ref{closure:F:sigma:under:symmetry}, it also the mean of $P_{\sigma}$. If $P$ has covariance matrix $\Sigma$ then, by Definition \ref{Psigma}, $P_{\sigma}$ has covariance matrix $\left( \sigma^2 + \frac{1-\sigma^2}{p+1} \right) \Sigma.$ For instance, if $P \sim \mathbf{N}(\mu, \Sigma)$, then $P_{\sigma} \sim \mathbf{N} \left( \mu, \left( \sigma^2 + \frac{1-\sigma^2}{p+1} \right) \Sigma \right).$
 \end{remark}

\subsection{Properties of the $\sigma$-simplicial depth}

 We study the properties of the simplex and distribution enlarged $\sigma$-simplicial depths. Proposition \ref{continuity:of:sigma:depth:dependent:vertices:wrt:sigma} examines them as a function of $\sigma$ while Theorems \ref{PropertiesS1} and \ref{PropertiesS2} concern a fixed $\sigma.$ The studied properties are those commonly studied for depth functions. In particular, affine invariance, maximality, monotonicity and vanishing at infinity, 
 which constitute the notion of multivariate statistical data depth \citep{Zuo-2000} and other desirable properties such as continuity, or merely upper-semicontinuity.

\begin{proposition} \label{continuity:of:sigma:depth:dependent:vertices:wrt:sigma}
Let $x \in \mathbb{R}^p.$
(i) Given a smooth distribution $P,$ $d_{\triangle_\sigma}(x; P)$ and $d_{P_{\sigma}}(x; P)$ are continuous as a function of $\sigma$. 
(ii) Given $P$ elliptically symmetric, $d_{P_{\sigma}}(x; P)$ is monotonically nondecreasing as a function of $\sigma.$ 
(iii) For general $P$, $d_{\triangle_\sigma}(x; P)$ is monotonically nondecreasing and continuous on the right as a function of $\sigma.$ 
\end{proposition}
(i) in the result above ensures that for any smooth $P,$ when $\sigma$ gets close to 1, the $\sigma$-simplicial depths have a similar behavior to that of the simplicial depth, and that their behavior varies in a continuous manner with $\sigma$.

Next results provide some fundamental properties of $\sigma$-simplicial depths. Upper semicontinuity is important in the last result of this subsection, while the continuity of the depth assures similar depth values for points close to each other.
\begin{theorem}\label{PropertiesS1}
The simplex enlarged $\sigma$-simplicial depth is (i) affine invariant, (ii) vanishes at infinity and is (iii) upper semicontinuous as a function of $x.$ Additionally, when computed with respect to a smooth distribution, the  simplex enlarged $\sigma$-simplicial depth is  (iv) continuous as a function of $x.$
\end{theorem}

The simplex enlarged $\sigma$-simplicial depth does not satisfy the maximality at the center property nor it is monotonically decreasing along rays from this point even for spherically symmetric distributions. However, for $p=1$ it is monotone for the points outside the convex hull of the support.

\begin{proposition} \label{proposition:monotonicityoutsidesupport}
Let $\sigma > 0,$ $P$ a distribution function on $\mathbb{R}$ and $S$ the convex hull of the support of $P.$ If $S \subset [a,b],$ $a \leq b$ then for either $x \leq y \leq a$ or $b \leq y \leq x$, we have that $d_{\triangle_\sigma}(x;P) \leq d_{\triangle_\sigma}(y;P)$.
\end{proposition}

\begin{theorem}\label{PropertiesS2}
The distribution enlarged $\sigma$-simplicial depth is (i) affine invariant, (ii) vanishes at infinity and is (iii) upper semicontinuous as a function of $x.$
Additionally, if computed with respect to a smooth distribution, it is also (iv) continuous as a function of $x.$ Furthermore, if computed with respect to a smooth distribution that is centrally symmetric, then the $\sigma$-simplicial depth satisfies the (v) maximality at the center property and is (vi) monotone nonincreasing along rays through the center.
\end{theorem}
Given a statistical data depth, $d,$ with respect to a distribution $P$ on  $\mathbb{R}^p$ and $\alpha > 0,$ the associated depth trimmed region is
$
 R_\alpha := \{ x \in \mathbb{R}^p : d(x; P) \ge \alpha \}.
$
Upper semicontinuity of the $\sigma$-simplicial depths implies that the corresponding trimmed regions are closed and there exists a maximizer. Further properties of the trimmed regions for the $\sigma$-simplicial depths follow from Theorem \ref{PropertiesS1}, Theorem \ref{PropertiesS2} and \citet[Theorem 3.1]{Zuo-2000-c}.

\begin{proposition}
\label{proposition:depthtrimmedregionsigmadepth}
The depth trimmed regions based on the simplex and distribution enlarged $\sigma$-simplicial depth are (i) affine equivariant, (ii) nested and (iii) compact. 
Moreover, for the distribution enlarged $\sigma$-simplicial depth, they are (iv) connected if $P$ is smooth and centrally symmetric.
\end{proposition}

\subsection{Definition and properties of the sample $\sigma$-simplicial depths}

We provide, in the following definitions, sample versions for the simplex enlarged (Definition \ref{sample:sigma:simplicial:depth:version:0}) and the distribution enlarged (Definition \ref{sample:sigma:simplicial:depth:version:1}) $\sigma$-simplicial depth functions.

\begin{definition}\label{sample:sigma:simplicial:depth:version:0}
Given $\sigma> 0,$ $p\geq 1$ and $n \geq p+1,$ the sample simplex enlarged \emph{$\sigma$-simplicial depth} of a point $x \in \mathbb{R}^p$ with respect to a distribution $P$ on $\mathbb{R}^p$ is 
\begin{equation*}
	d_{\triangle_{\sigma},n}(x; P) := \binom{n}{p+1}^{-1} \sum_{1 \leq i_1 < \dots < i_{p+1} \leq n} \mathbf{I} ( x \in \triangle_{\sigma} [ X_{i_1}, \dots, X_{i_{p+1}} ] ), 
\end{equation*}
where $X_1, \dots, X_{n}$ is a sample of random draws taken from $P.$
\end{definition}

\begin{definition}\label{sample:sigma:simplicial:depth:version:1}
Given $\sigma> 0,$ $p\geq 1,$ $n \geq (p+1)^2$ and $k:=k(n,p)$  the greatest integer less than or equal to $n/(p+1),$ the sample distribution enlarged \emph{$\sigma$-simplicial depth} of a point $x \in \mathbb{R}^p$ with respect to a distribution $P$ on $\mathbb{R}^p$ is 
$$ d_{P_{\sigma}, k}(x; P):= \binom{k}{p+1}^{-1} \sum_{1 \leq i_1 < \dots < i_{p+1} \leq k} \mathbf{I} ( x \in \triangle [ Y_{i_1}, \dots, Y_{i_{p+1}} ] ),
$$
where $Y_i = \sigma\cdot(X_{1+(i-1)(p+1)} - \bar{X}_i) + \bar{X}_i,$  $\bar{X}_i = \sum_{j=1}^{p+1} X_{j+(i-1)(p+1)}/(p+1)$ for $i=1, \dots, k$ and  $X_1, \dots, X_{n}$ is a sample of random draws taken from $P.$ 
\end{definition}
Note that $Y_1, \dots, Y_{k}$ in Definition \ref{sample:sigma:simplicial:depth:version:1} are random draws from $P_{\sigma},$ each obtained as a linear combinations of $p+1$ random draws from $P$. Thus, $d_{P_{\sigma}, k}(\cdot; P)=d_{S,k}(\cdot,P_\sigma),$ with $d_{S,k}(\cdot,P_\sigma)$ denoting the sample simplicial depth, as in \eqref{ss}, based on $k$ random draws from the distribution $P_{\sigma}$. Taking $\sigma=1,$ we obtain $d_{P_{\sigma}, k}(\cdot; P)=d_{S,k}(\cdot,P),$  retrieving the sample simplicial depth, but based only on $k$ of the $n$ random draws $X_1, \dots, X_{n}$. 

As for the theoretical $\sigma$-simplicial depths, their sample versions also take value zero for elements of the space  outside a region that depends on the selected $\sigma.$ This does not pose any problem in applications since it is always possible to choose $\sigma$ such that the sample $\sigma$-simplicial depths assign positive value to the region of interest.

Next results study the properties of the estimators of the $\sigma$-simplicial depths.

\begin{proposition} \label{proposition:sample:sigma:depth:pathwise:nondecreasing}
For any distribution $P$ on $\mathbb{R}^p$ and any $x \in \mathbb{R}^p,$ the function 
\begin{equation*}
\begin{split}
	[0, \infty) & \rightarrow [0,1] \\
	\sigma & \mapsto d_{\triangle_{\sigma}, n}(x ; P)
\end{split}
\end{equation*}
is monotonically nondecreasing for each realization of the random variables $X_1, \dots, X_{n},$ drawn with distribution $P.$
\end{proposition}
The above result is  in general not valid for the estimator of the distribution enlarged $\sigma$-simplicial depth.

For any distribution $P$ on $\mathbb{R}^p,$ $d_{\triangle_{\sigma},n}(\cdot; P)$ is a $U$-statistic for the estimation of $d_{\triangle_\sigma}(\cdot; P)$ while $d_{P_{\sigma}, k}(\cdot; P)$ is a $U$-statistic with respect to $Y_{i}$'s for the estimation of $d_{P_{\sigma}}(\cdot; P).$ Hence, the theoretical results for the empirical $\sigma$-simplicial depths can be derived by the properties of U-processes \citep{Korolyuk-2013}. In particular, the sample depths converge almost surely to their population counterparts. Moreover, since the set of simplices in $\mathbb{R}^p$ forms a VC-class \citep{Arcones-1993}, this convergence can be made uniform over $\mathbb{R}^p$, as it is shown in  next theorem. This result is fundamental in practical applications since it ensures that the empirical $\sigma$-depths are a good approximation of their population counterparts as the sample size increases.

\begin{theorem} \label{uniform:consistency}
For any distribution $P$ on $\mathbb{R}^p,$ we have that 
\begin{itemize}
\item[]$\sup_{ x \in \mathbb{R}^p } \abs{ d_{\triangle_{\sigma},n} (x; P) - d_{\triangle_\sigma} (x; P) } \xrightarrow[ n \to \infty ]{} 0$ almost surely and
\item[]$\sup_{ x \in \mathbb{R}^p } \abs{ d_{P_{\sigma},k(n,p)} (x; P) - d_{P_{\sigma}} (x; P) } \xrightarrow[ n \to \infty ]{} 0$ almost surely.
\end{itemize}
\end{theorem}
Using the uniform convergence of the sample simplicial depths to their population counterpart, the convergence of the maximizer can also be established. The maximizer of the distribution enlarged $\sigma$-simplicial depth corresponds, for centrally symmetric distributions, to the point of central symmetry (Theorem \ref{PropertiesS2}). The following corollary ensures in particular that the maximizer of $d_{P_\sigma,k}$ is a good approximation of this point as the sample size increases.

\begin{corollary} \label{convergence:of:the:sample:median}
The following statement is satisfied for $(d,d_n)$ equal to either $(d_{\triangle_\sigma},d_{\triangle_{\sigma},n})$ or $(d_{P_{\sigma}},d_{P_{\sigma},k(n,p)}).$ For any distribution $P$ on $\mathbb{R}^p$ such that $d(\cdot; P)$ is uniquely maximized, we have that $\mu_n \rightarrow \mu$ almost surely, where $\mu$ and $\mu_n$ are maximizers of $d(\cdot; P)$ and $d_{n}(\cdot; P),$ respectively.
\end{corollary}
In the following theorem we establish the asymptotic normality of the empirical process defined by the introduced sample $\sigma$-simplicial depths as a function of $x.$ $\ell^{\infty} (\mathbb{R}^p)$ is the space of all real bounded functions on $\mathbb{R}^p$ and the convergence in law in $\ell^{\infty} (\mathbb{R}^p)$ (denoted by $\rightsquigarrow$) is in the sense of Hoffmann-J{\o}rgensen \citep{Masse-2002}.
\begin{theorem} \label{CLT}
For any distribution $P$ on $\mathbb{R}^p,$
\begin{itemize}
\item[(i)] $\{ \sqrt{n} \, ( d_{\triangle_\sigma,n}(x; P) - d_{\triangle_\sigma} (x; P) ) : x \in \mathbb{R}^p \} \underset{n \to \infty}{\rightsquigarrow} \{ (p+1) G_{P}^{\triangle_{\sigma}}(x) : x \in \mathbb{R}^p \}$ and
\item[(ii)] $\{ \sqrt{k(n,p)} \, ( d_{P_{\sigma},k(n,p)} (x; P) - d_{P_{\sigma}} (x; P) ) : x \in \mathbb{R}^p \} \underset{n \to \infty}{\rightsquigarrow} \{ (p+1) G_{P_{\sigma}}(x) : x \in \mathbb{R}^p \}.$
\end{itemize}
$G_{P}^{\triangle_{\sigma}}$ and $G_{P_{\sigma}}$ are centered Gaussian process with covariance function
\begin{equation*}
\mathbb{E} [ G(x) G(y) ] = \int g_x (z) g_y (z) \, \d Q(z) - \int g_x (z) \, \d Q(z) \int g_y (z) \, \d Q(z)
\end{equation*} 
where in (i) $g_x (z) = \int \mathbf{I} ( x \in \triangle_{\sigma} [ x_1, \dots, x_p, z ] ) \, \d P(x_1) \dots \d P(x_p)$ and $Q=P$, and in (ii) $g_x (z) = \int \mathbf{I} ( x \in \triangle [ x_1, \dots, x_p, z ] ) \, \d P_{\sigma}(x_1) \dots \d P_{\sigma}(x_p)$ and $Q=P_{\sigma}$.
\end{theorem}
Theorem \ref{CLT} implies that, roughly speaking, the errors in the approximation in Theorem \ref{uniform:consistency} are asymptotically normal with mean $0$ and variance and covariance specified by $\mathbb{E} [ G(x) G(y) ]$, for $x,y \in \mathbb{R}^p$. In the next section we study the finite-sample performance of the empirical $\sigma$-depths by means of Monte Carlo simulations. In Section \ref{data_analysis}, we use them to perform classification on four datasets. 

\section{Simulations}\label{Simulations} 
We consider below three simulation settings to compare our estimators among them and with the sample simplicial depth, the refined halfspace depth \citep{Einmahl-2015} and the illumination depth \citep{Nagy-2021}. The comparison is made making use of supervised classification. Unless when stated otherwise, the classification is performed using the linear DD-classifier as in \citet{Li-2012}. For more general classification algorithms see \citet{Cuesta-2017, Hubert-2017}.

The refined halfspace depth is based on refining the sample halfspace depth outside and in proximity of the boundaries of the convex hull of each unidimensional projection by means of extreme value statistics. The procedure requires to select two parameters $r$ ($k$ in their notation) and $\gamma$, where $r$ is the sub-sample size of the refinement and $\gamma$ controls the speed of decrease outside the central region, estimated using the sample. Specifically, the parameter $\gamma$ is estimated according to the recommendations at pag.\ 2748 of \citet{Einmahl-2015} while $r$ is set in a discrete range of values, and the refined halfspace depth is computed using $500$ random projections following \citet{Cuesta-2008} as in the R code provided by the authors. Illumination depth coincides with halfspace depth in the halfspace depth trimmed region $R_\alpha$ of level $\alpha > 0$, while the depth of a point $x \notin R_\alpha$ is given as a fraction of $\alpha$ that depends on the ratio between the volume of $R_\alpha$ and the volume of the convex hull of $R_\alpha \cup \{x\}$. In our simulations we consider values of $\alpha$ in a grid from $0.02$ to $0.2$ and use the R functions available in the supplementary material of their paper. 
 
{\it Simulation 1.} 
We study empirically our proposal under a two-class classification scenario inspired by \citet[Section 3.2]{Einmahl-2015}. Consider the two following samples: $S_n^{(1)}=\{ X_1^{(1)}, \dots, X_{n}^{(1)} \},$ drawing $n$ independent and identically distributed random variables from a distribution $P^{(1)},$ and $S_m^{(2)}=\{ X_1^{(2)}, \dots, X_{m}^{(2)} \},$ drawing $m$ independent and identically distributed random variables from a distribution $P^{(2)}.$ 
We make use of the following settings: (i) difference in location, $X_1^{(2)} \stackrel{d}{=} X_1^{(1)} + (c, c)^\top,$  (ii) difference in scale, $X_1^{(2)} \stackrel{d}{=} 3 X_1^{(1)},$  and (iii) difference in both, scale and location, $X_1^{(2)} \stackrel{d}{=} 3 X_1^{(1)} + (c,c )^\top.$ We do each of these three settings for two types of distributions, that we refer to as normal and elliptical. For the normal type, $P^{(1)}$ is a standard bivariate normal distribution and  $c=2.$ For the elliptical type,  $P^{(1)}$ has the density function 
\begin{equation*}
	f(x,y)= \begin{cases}
		\frac{3}{4 \pi} r_0^4 \left( 1+r_0^6 \right)^{-\frac{3}{2}} \quad & \frac{x^2}{4} + y^2 < r_0^2 \\
		\frac{3 \left( \frac{x^2}{4} + y^2 \right)^2}{ 4 \pi \left(1 + \left( \frac{x^2}{4} + y^2 \right)^3 \right)^{ \frac{3}{2} } }  \quad & \frac{x^2}{4} + y^2 \geq r_0^2,
	\end{cases}
\end{equation*}
where $r_0 \approx 1.2481,$ and $c=4.$
For each of the depth functions  and  each choice of $P^{(1)}$ and $P^{(2)},$ we train the linear DD-classifier on the samples $S_n^{(1)}$ and $S_m^{(2)}$ with $n=m=500.$   We generate $5000$ new observations ($2500$ from $P^{(1)}$ and $2500$ from $P^{(2)}$). 

We compute the misclassification rates, according to the linear classifier, for the new samples outsiders of both $S_n^{(1)}$ and $S_m^{(2)}.$ That is, among the $5000$ new observations, we consider as test sample the observations outside both the convex hull of $S_n^{(1)}$ and the convex hull of $S_m^{(2)}.$ We repeat this experiment $100$ times.  
In Figure \ref{boxplots:sigma:simplicial:version:0:out} 
\begin{figure}[t]
	\includegraphics[width=\linewidth,height=.5\linewidth]{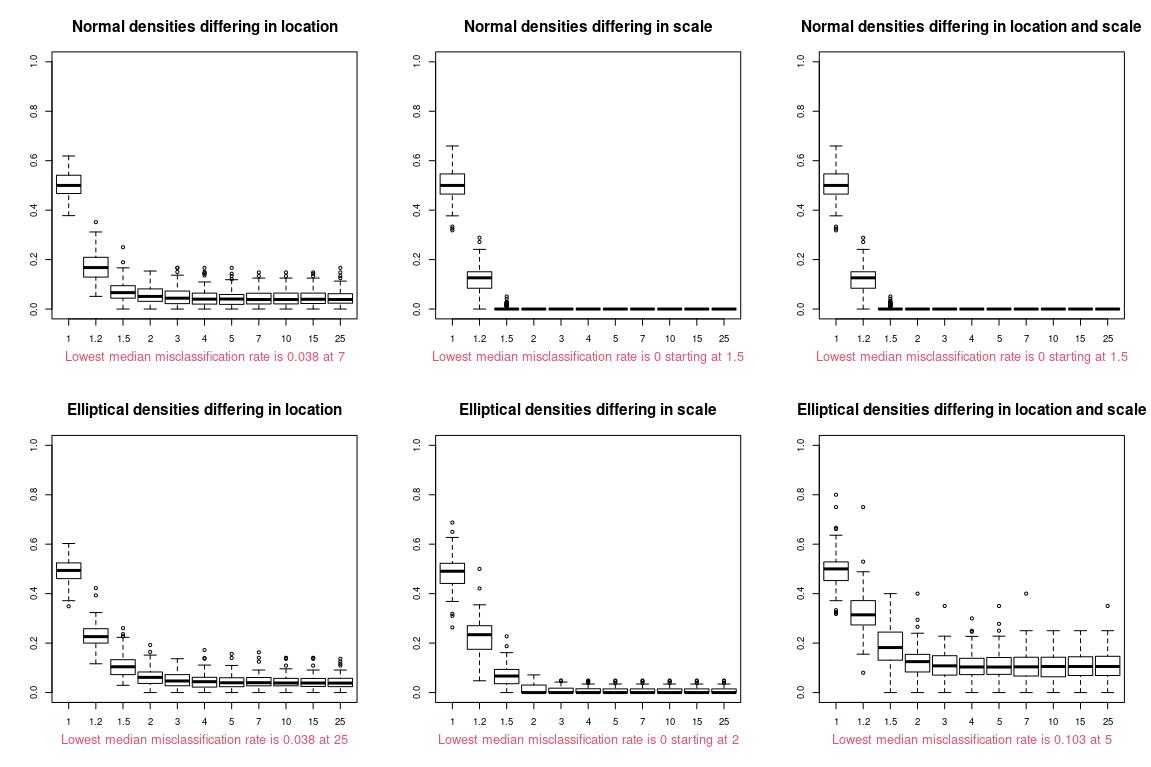}
	\caption{Boxplots of $100$ misclassification rates of the outsiders for $d_{\triangle_{\sigma},n}$ with $\sigma\in\{1.2, 1.5, 2, 3, 4, 5, 7, 10, 15, 25\}$ and the sample simplicial depth ($\sigma=1$), using the linear DD-plot classifier.}
 	\label{boxplots:sigma:simplicial:version:0:out}
\end{figure}
we illustrate the obtained misclassification rates for the sample simplex enlarged $\sigma$-simplicial depth, $d_{\triangle_{\sigma},n}$,  for different values of $\sigma,$ in the range from from 1.2 to 25, and for the sample simplicial depth ($\sigma=1$). Figure \ref{boxplots:sigma:simplicial:version:1:out} in Section \ref{Asim} of the supplementary material contains the misclassification rates for the sample distribution enlarged $\sigma$-simplicial depth, whose performance is slightly worse. 
Figures \ref{boxplots:refined:halfspace:out} and \ref{boxplots:illumination_depth_out} 
show the results for the refined halfspace depth and the illumination depth respectively. Comparing the results it is evident that the misclassification rates obtained by $d_{\triangle_{\sigma},n}$ are dramatically lower than those of the refined halfspace depth, except when there is a big difference in location, in which case the refined halfspace depth is only slightly worse. The illumination depth and $d_{\triangle_{\sigma},n}$ perform similarly except when $\alpha$ is small and the difference is in scale; in these cases the illumination depth performs badly. 
\begin{figure}[t]
	\includegraphics[width=\linewidth,height=.5\linewidth]{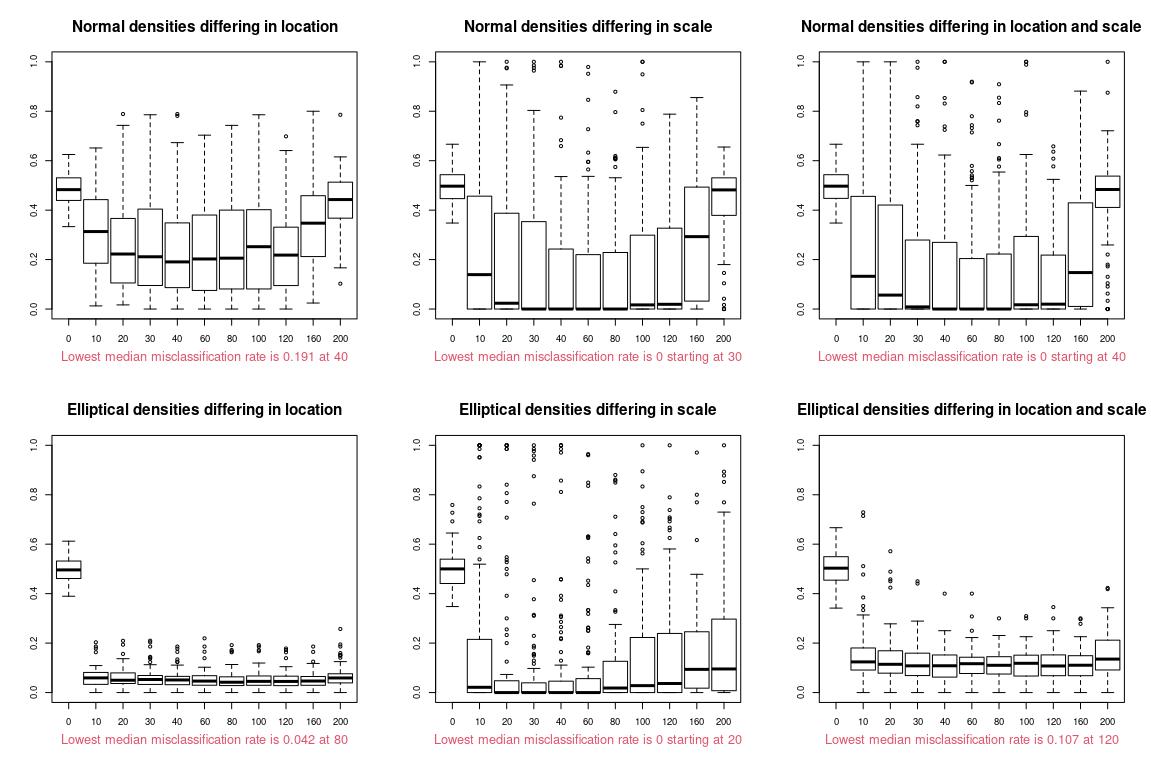}
	\caption{Boxplots of $100$ misclassification rates of the outsiders for the refined halfspace depth with $r\in\{10, 20, 30, 40, 60, 80, 100, 120, 160, 200\}$ and the sample halfspace depth ($r=0$), using the linear DD-plot classifier.}
 	\label{boxplots:refined:halfspace:out}
\end{figure}

Section \ref{Asim}, in the supplementary material, contains the misclassification rate for the whole sample, all the 5000 new observations are considered as test sample, for the different studied depth functions (see Figures \ref{TtriangleT}, \ref{figure:refined:halfspace}, \ref{figure:illumination_depth} and \ref{BskT}). This section also contains misclassification rates for other classification methods, namely, linear discriminant analysis \citep[LDA]{Fisher-1936}, quadratic discriminant analysis \citep[QDA]{Hastie-2001}, and k-nearest neighbors \citep[$k$-NN]{Cover-1967}, with $k$ varying from one to ten. See also there, Table \ref{table:misclassification_rates_whole_sample} for the whole sample and Table \ref{table:misclassification_rates_outsiders} for the outsiders. 

Since the misclassification rates are stable over a wide range of values of $\sigma$, we suggest to take the smallest $\sigma$ such that all points have positive depth with respect to at least one of the two samples. Only if there are still many ties it might be worthwhile to choose a bigger $\sigma$. See also Section \ref{data_analysis} below.
\begin{figure}[t]
	\includegraphics[width=\linewidth,height=.5\linewidth]{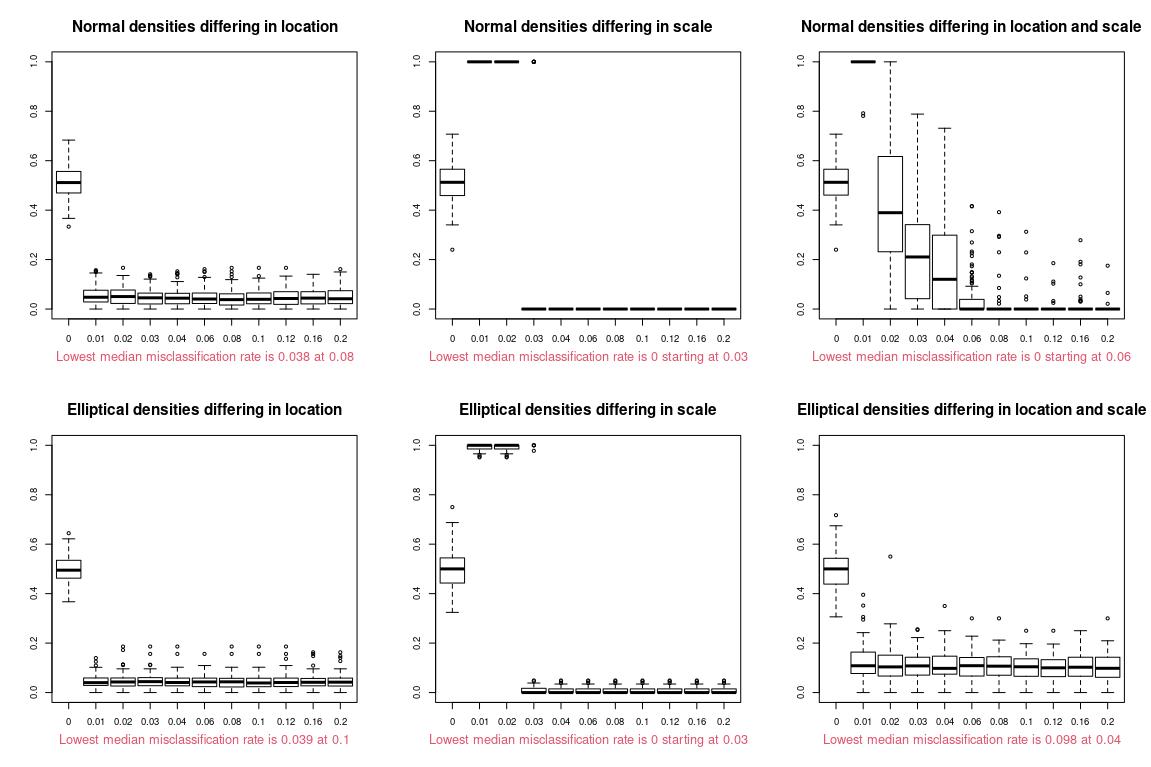}
	\caption{Boxplots of $100$ misclassification rates of the outsiders for the sample illumination depth with $\alpha \in \{0.01, 0.02, 0.03, 0.04, 0.06, 0.08, 0.10, 0.12, 0.16, 0.20 \}$ and the sample halfspace depth ($\alpha=0$). Linear DD-plot classifier.}
 	\label{boxplots:illumination_depth_out}
\end{figure}

{\it Simulation 2.} We perform a simulation under the scenario of four bivariate independent normal distributions with identity covariance matrix. We denote them by $P^{(1)}, P^{(2)}, P^{(3)}, P^{(4)}.$ We take $P^{(1)}$ with mean $(-4,0),$ $P^{(2)}$ with mean $(-\delta,0)$ where  $\delta \in (0,4)$ and  $P^{(4)}$ with mean $(4,0).$ We consider two scenarios: 
(i) symmetric distributions:   $P^{(3)}$ with mean $(\delta,0)$; and (ii) asymmetric distributions:  $P^{(3)}$ with mean $(2,0)$.
The perfect classifier would result on closeness of  $P^{(2)}$  to $P^{(1)}$ and of  $P^{(3)}$ to $P^{(4)}.$

We generate $500$ training samples from $P^{(1)}$ and another $500$ from $P^{(4)}$ and $2500$ test samples from  $P^{(2)}$ and another  $2500$ from $P^{(3)}$. The procedure is repeated $100$ times for each $\delta\in\{.1, .2, .3, \ldots, 3.9 \}$. The mean of the misclassification rates using the sample simplicial depth and the sample simplex enlarged $\sigma$-simplicial depth, $d_{\triangle_{\sigma},n},$ for different values of $\sigma$ are computed and plotted as a function of $\delta$ in Figure \ref{figure:data:normals}. 
\begin{figure}[!htb]
\begin{center}
\includegraphics[width=.49\linewidth]{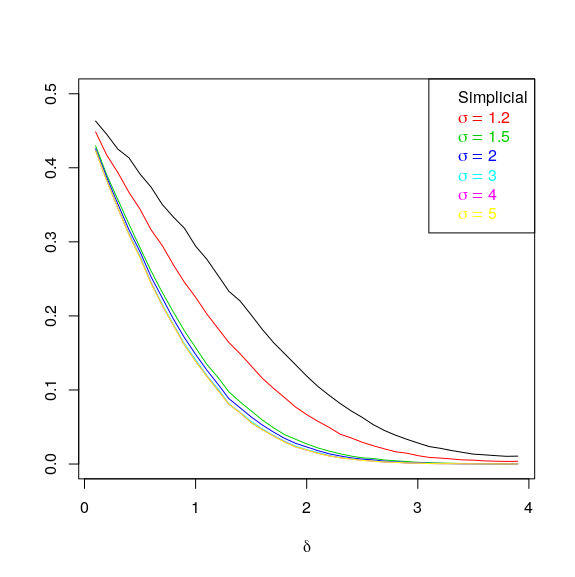}
\includegraphics[width=.49\linewidth]{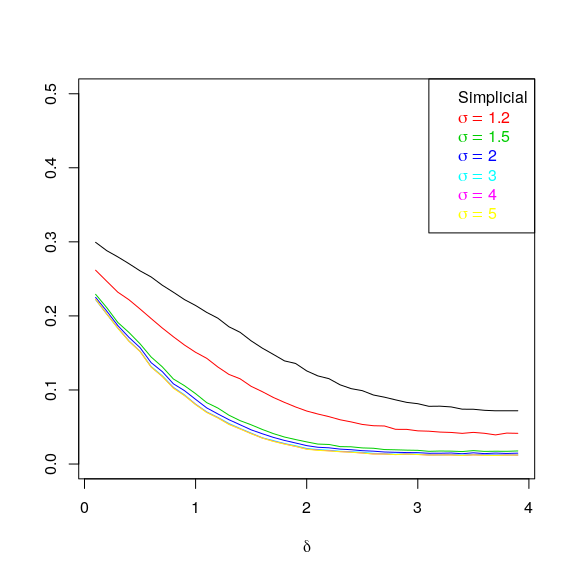}
\end{center}
\caption{Mean, as a function of $\delta$, of misclassification rates over $100$ times  using the sample simplicial depth and $d_{\triangle_{\sigma},n}$ for $\sigma\in \{1.2, 1.5, 2, 3, 4, 5\}$. On the left symmetric distributions are given, on the right asymmetric ones.}
\label{figure:data:normals}
\end{figure}
Clearly, $d_{\triangle_{\sigma},n}$ outperforms the sample simplicial depth for all values of $\delta$. Furthermore, the results are similar for any $\sigma\geq 1.5,$ which shows that this procedure is stable with respect to the choice of $\sigma$.

{\it Simulation 3.} Another drawback of sample simplicial depth is that it is not able to differentiate among points on the boundary of the convex hull of the sample. The same problematic occurs for sample halfspace depth, which assigns depth value of $1/n$ to all these points. It is worth mentioning that for sample halfspace depth several ties may also occur for observations in the inner part of the convex hull whereas for simplicial depth ties primarily occur only for observations on the boundary of the convex hull. We propose to break ties using $\sigma$-simplicial depth with $\sigma>1$.

 To test our methodology, we draw $n \in \{50, 200 \}$ observations from the standard bivariate normal distribution and compare the correct ranking of the observations on the boundary of the convex hull with the ranking given by the sample simplex enlarged $\sigma$-simplicial depth, $d_{\triangle_{\sigma},n}$. We repeat the experiment $100$ times. Figure \ref{figure:correlation} contains boxplots of the (Spearman) correlation between the correct ranking and the ranking given by $d_{\triangle_{\sigma},n}$ for $\sigma\in \{5, 10, 15, 20, 25\}$. We compare the performance of $d_{\triangle_{\sigma},n}$ with refined halfspace depth for $r/n \in \{ 0.05, 0.1, 0.15, 0.2, 0.25 \}$ and illumination depth, where we use a fraction $c \in \{0.3, 0.4, 0.5, 0.6, 0.7\}$ of the sample points to build the halfspace central region (see Section 5.1 of \citet{Nagy-2021}). From Figure \ref{figure:correlation} we observe that $d_{\triangle_{\sigma},n}$ outperforms refined halfspace depth and it is slightly preferable to illumination depth. Finally, we notice that the results are stable with respect to the choice of $\sigma$.
\begin{figure}[!htb]
\begin{center}
\includegraphics[width=\linewidth]{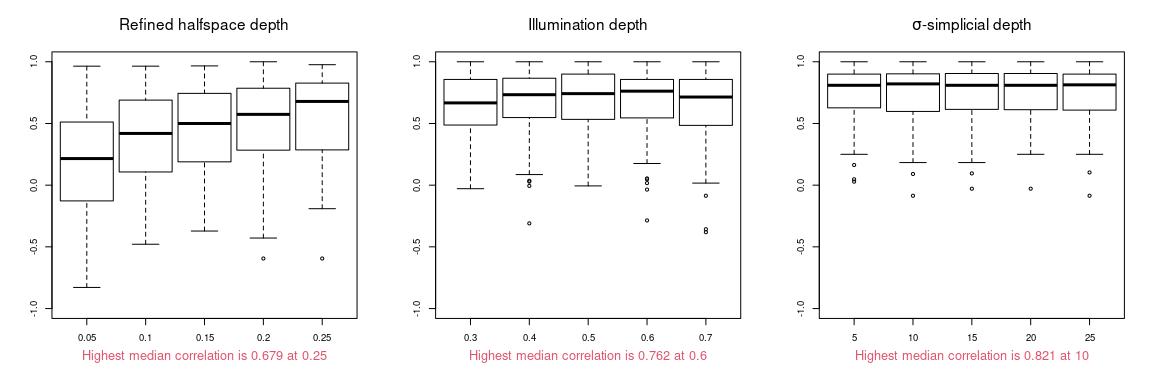}
\includegraphics[width=\linewidth]{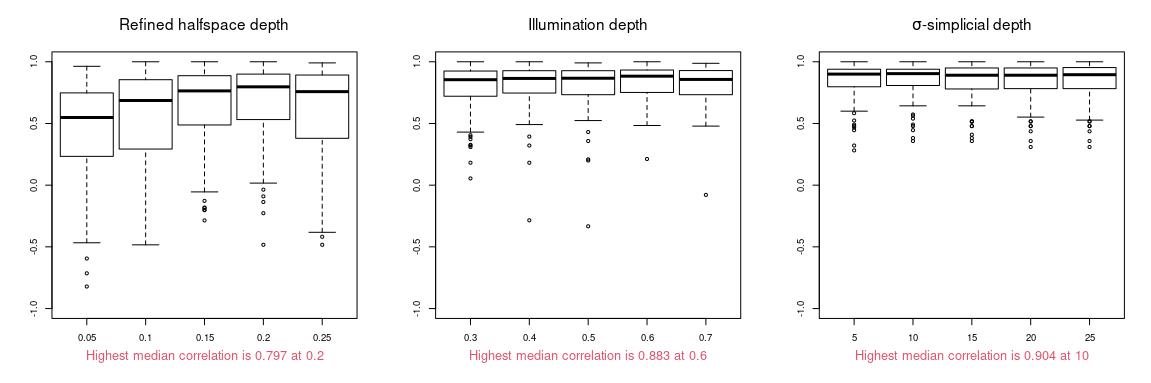}
\end{center}
\caption{Boxplots of $100$ (Spearman) correlation coefficients between the correct ranking of sample points on the boundary of the convex hull and the rankings based on the refined halfspace depth with $r/n \in \{ 0.05, 0.1, 0.15, 0.2, 0.25 \}$ (left), illumination depth using a fraction $c \in \{0.3, 0.4, 0.5, 0.6, 0.7\}$ of sample points to build the halfspace central region (center), and $d_{\triangle_{\sigma},n}$ with $\sigma \in \{ 5, 10, 15, 20, 25 \}$ (right). The sample size is $n=50$ in the first row and $n=200$ in the second row.}
\label{figure:correlation}
\end{figure}

\section{Data analysis}
\label{data_analysis}

We consider a supervised classification framework with two or more classes. We examine four datasets.

{\it 1.} AIS dataset \citep{Cook-2009} contains physical and blood measurements from high performance athletes (100 female and 102 male) at the Australian Institute of Sports (AIS).  We perform principal component analysis \citep[PCA]{Hotelling-1936} and use the first two principal components as measurements. We consider the athletes' sex as the class.

{\it 2.} Chaetocnema dataset \citep{Lubischew-1962} includes anatomical measures of 74 males of flea beetles, of three different genera: Ch.\ Concinna (21 cases), Ch.\ Heikertingeri (31 cases) and Ch.\ Heptapotamica (22 cases).  We consider three measurements for each class: (i) the front angle of the aedeagus, (ii) the maximal width of the head between the external edges of the eyes, and (iii) the aedeagus width from the side.

{\it 3.} Haltica dataset  \citep{Lubischew-1962} includes anatomical measures of 39 Haltica flea beetles specimens which belong to two different species: Haltica oleracea and Haltica carduorum. We consider two measurements: (i) the distance of the transverse groove from the posterior border of the prothorax and (ii) the length of the elytram.

{\it 4.} Iris dataset \citep{Fisher-1936} consists of $150$ observations equally divided among three classes (Iris Setosa, Iris Versicolour, and Iris Virginica) with four measurements each (sepal length, sepal width, petal length, and petal width).

We perform supervised classification on them making use of six different methodologies. Three are based on statistical data depth and make use of the DD-classifier. The others are the classical  LDA, QDA and $k$-NN methodologies. For the depth based procedures, we make use of the proposed sample simplex enlarged $\sigma$-simplicial depth, $d_{\triangle_{\sigma},n},$ and of the refined halfspace and  illumination depths. For the parameters, we take $\sigma$ between 1 and 1000 for $d_{\triangle_{\sigma},n}$, $r \in \{1,2, \dots, 10\}$ for  refined halfspace depth, $\alpha \in \{ 0.01,0.02,\dots, 0.2\}$ for illumination depth and $k \in \{ 1, \dots, 10 \}$ for $k$-NN.

To compare the performance of the different methods, we divide each of the above datasets into training and test sets. We do this by taking the first part of each class as training and the rest as test. When the size of the class is odd, one more datum is used for training. We report in Table \ref{table:correct_classification_using_datasets} the number of test data correctly classified.
\begin{table}[!htbp]
  \caption{\label{table:correct_classification_using_datasets} The first three columns correspond to the name of the employed datasets and the sizes of the used training and test data. The rest of the columns correspond to the number of test data points correctly classified using $d_{\triangle_{\sigma},n}$, refined halfspace depth, illumination depth, LDA, QDA, and $k$-NN. Column $d_{\triangle_{\sigma},n}^{\tiny\mbox{CV}}$ contains results using the value of $\sigma$ obtained via cross-validation. The other columns contain the best results across the parameters.}
\vspace{0.5cm}  
 \centering
\small\addtolength{\tabcolsep}{-3pt}
\begin{tabular}{p{2.4cm}p{1.3cm}p{1.3cm} p{1.3cm}  p{1.3cm} p{1.3cm}  p{1.3cm}  p{1.3cm}  p{1.3cm}  p{1.2cm}}
Dataset& Train s. & Test s. & $d_{\triangle_{\sigma},n}^{\tiny\mbox{CV}}$ & $d_{\triangle_{\sigma},n}$ & R.\ half.\ & Illumin.\ & LDA & QDA & $k$-NN \\
\hline
AIS & 101 & 101 & 84 & {\bf{89}} & 84 & {\bf{89}} & 87 & 86 & 82 \\
Chaetocnema & 38 & 36 & {\bf{36}} & {\bf{36}} & 24 & {\bf{36}} & {\bf{36}} & {\bf{36}} & {\bf{36}} \\
Haltica & 20 & 19 & 17 & {\bf{18}} & 15 & 17 & 17 & 15 & 16 \\
Iris & 75 & 75 & {\bf{74}} & {\bf{74}} & 54 & 73 & 72 & 73 & 71 \\
\end{tabular}
\end{table}
Columns five to seven and ten contain the best results across the parameters. They correspond to  $d_{\triangle_{\sigma},n}$, refined halfspace depth, illumination depth and $k$-NN.  LDA and QDA do not depend on parameters and the corresponding results are reported in columns eight and nine.

\begin{figure}[!htbp]
\begin{center}
  \includegraphics[width=.44\linewidth]{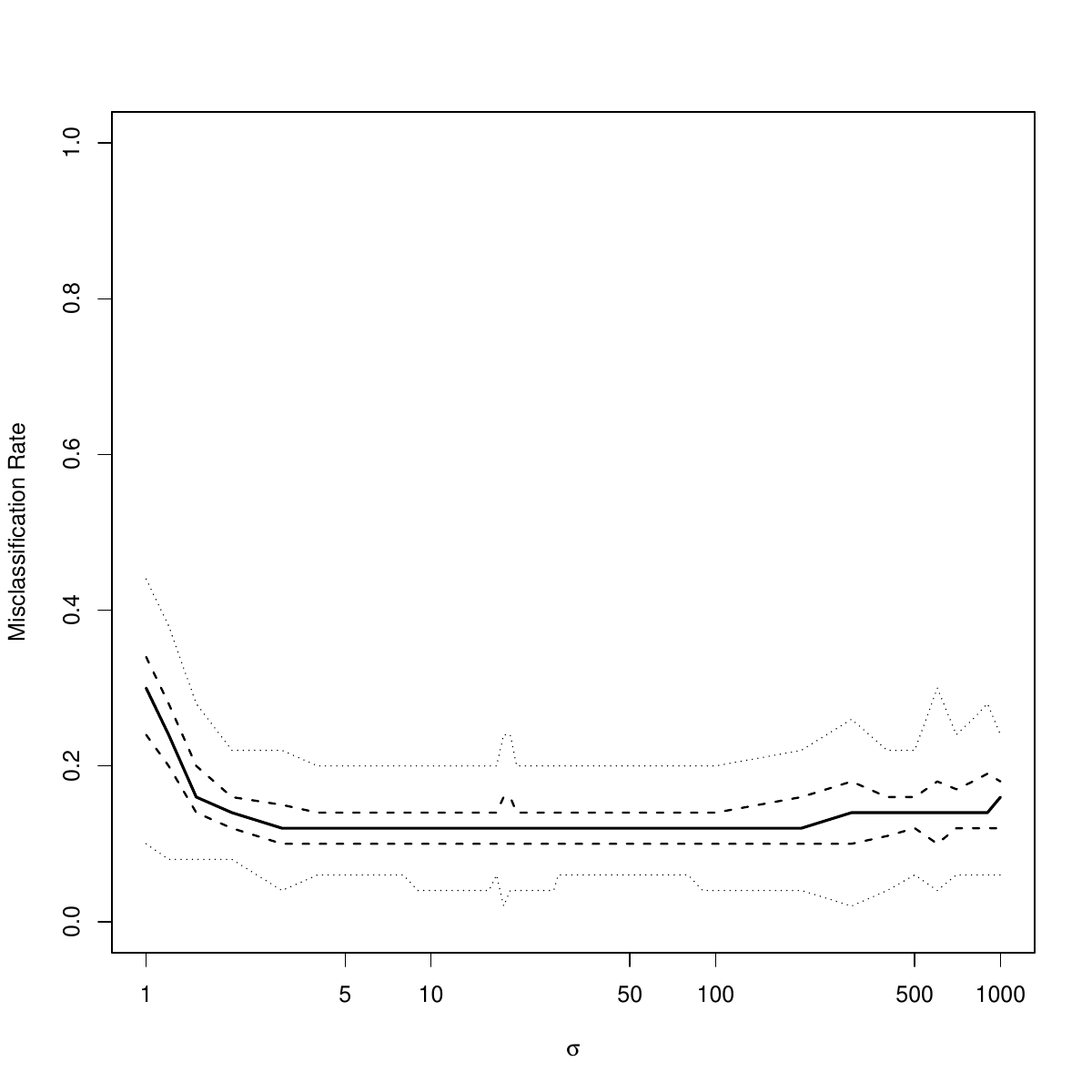}
  \includegraphics[width=.44\linewidth]{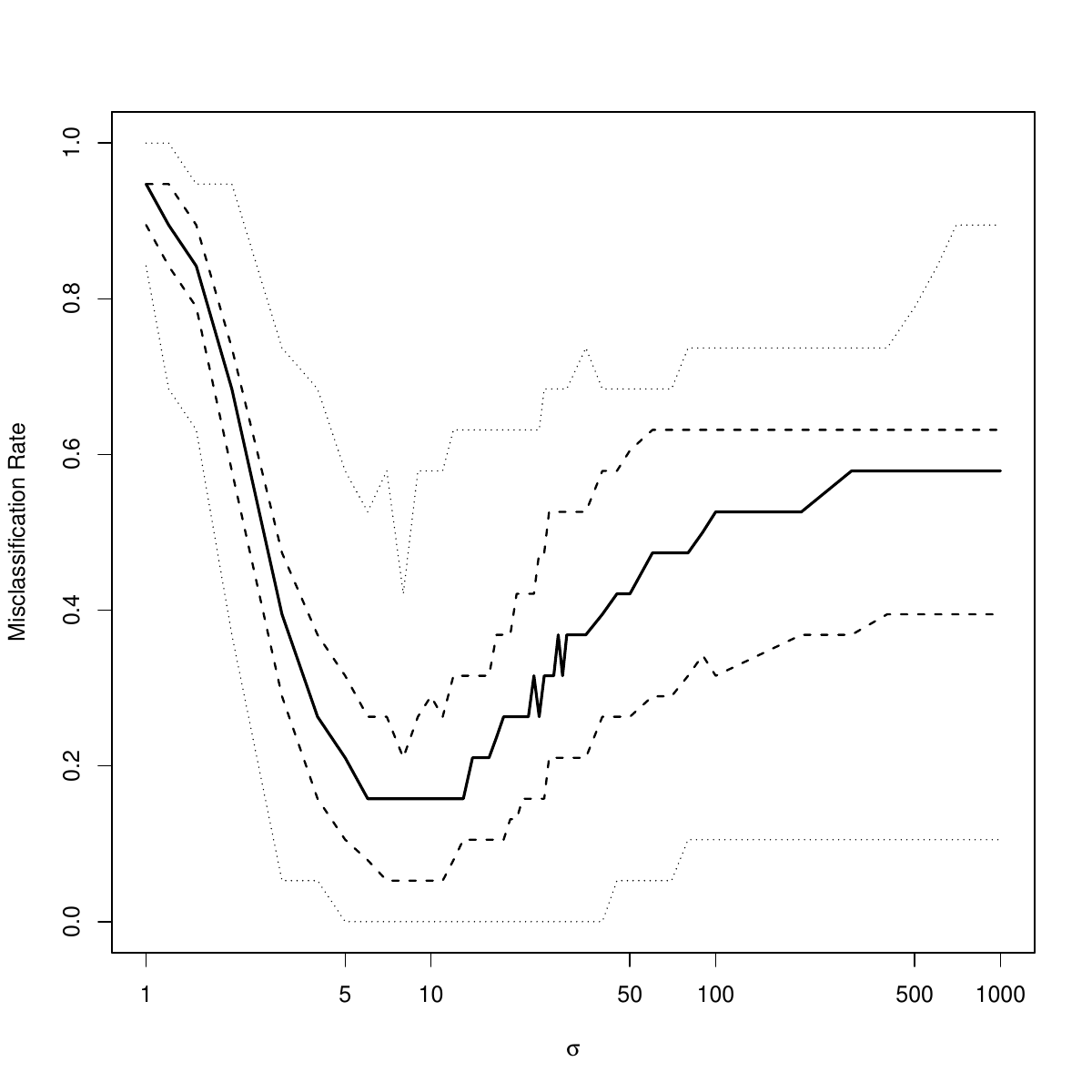}
  \includegraphics[width=.44\linewidth]{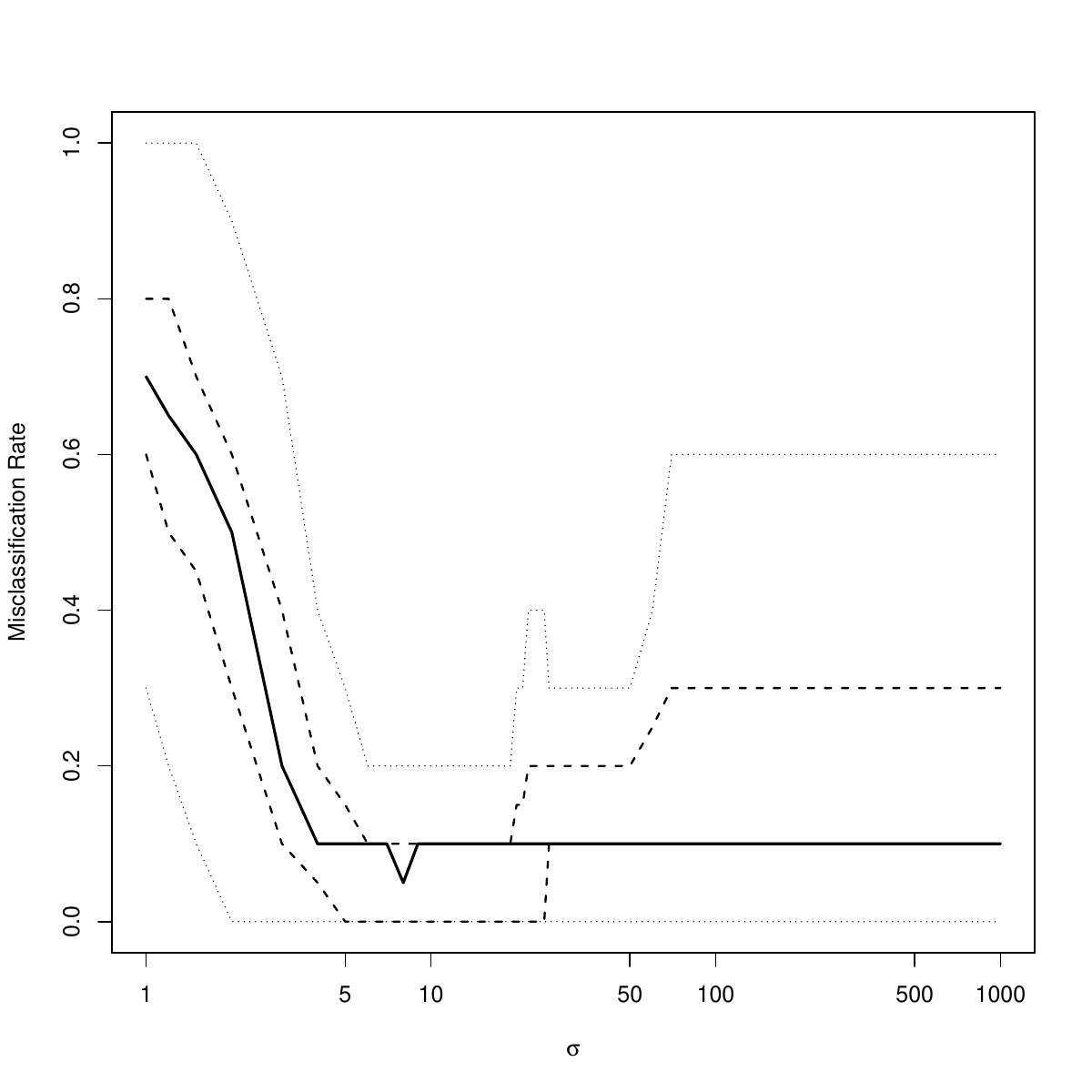}
  \includegraphics[width=.44\linewidth]{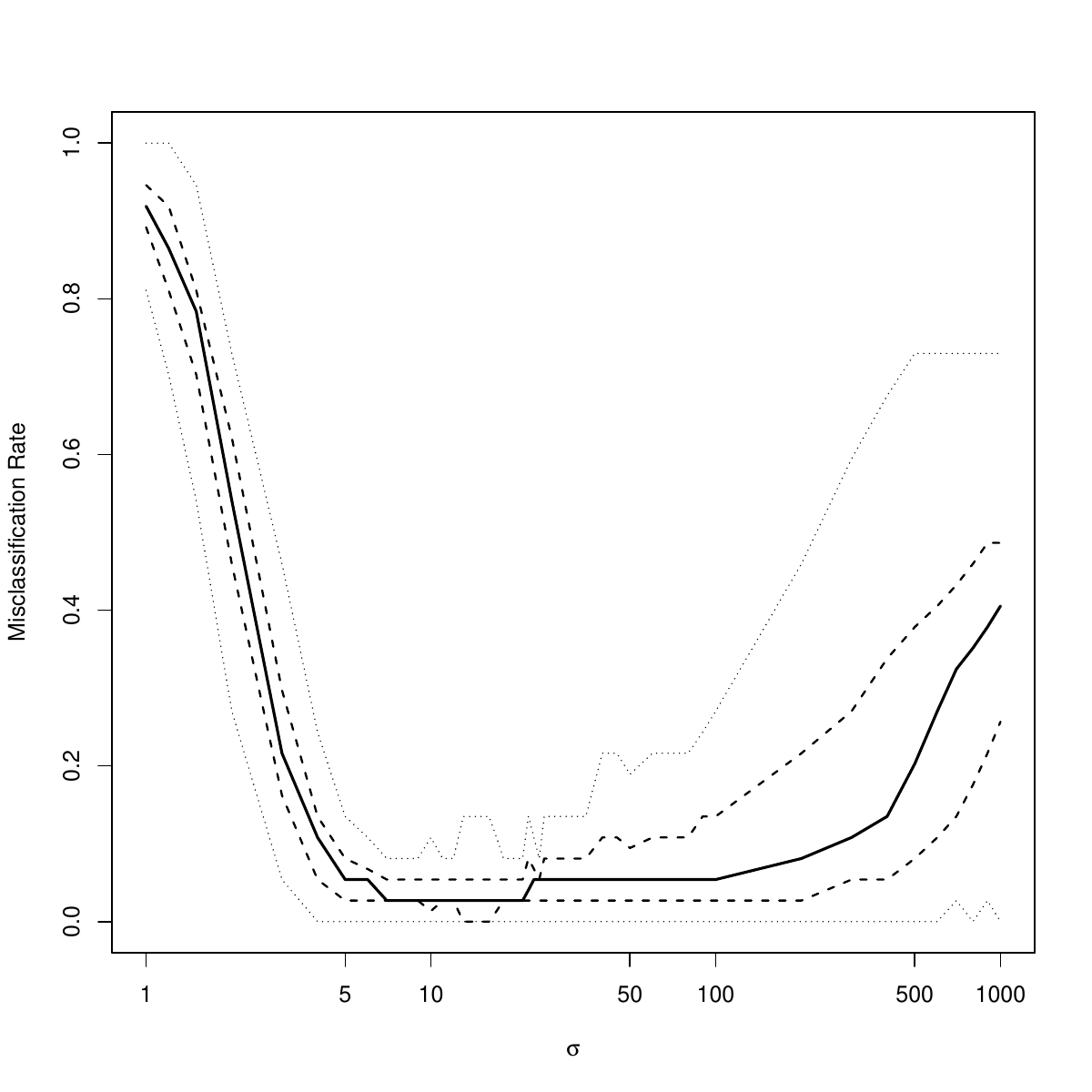}
\end{center}
\caption{Median (solid line), $25\%$ and $75\%$ quantiles (dashed lines) and  whiskers of the boxplot (dotted lines) of the misclassification rates using $d_{\triangle_{\sigma},n}$ as a function of $\sigma$, over $100$ cross-validation resamplings, for AIS dataset (top left), Chaetocnema dataset (top right), Haltica dataset (bottom left), and Iris dataset (bottom right).}
\label{figure:cv_iris_and_chaetocnema}
\end{figure}

For $d_{\triangle_{\sigma},n}$, we also report in the fourth column the results obtained by choosing $\sigma$ via cross-validation on the training dataset. This is done by randomly splitting the training set into two parts: one is a secondary training set and the other is a secondary test set. This procedure is repeated $100$ times and the boxplots of the misclassification rates are reported as a function of $\sigma$ in Figure \ref{figure:cv_iris_and_chaetocnema}. The resulting median curve is plotted in a solid line, the first and third quartile in dashed lines and the boxplot's whiskers in dotted lines. We observe a similar pattern among the plots of the four datasets, as $\sigma$ increases, there is a quick decline of the misclassification rate followed by a constant period. Then, the misclassification rate  increases again as a function of $\sigma$ for some of the cases,  but at a lower rate. When multiple values of $\sigma$ yield the same median misclassification rate, we choose the smallest $\sigma$ among them. For the refined halfspace and illumination depths, we do not run cross-validation, as these methods require a larger sample size to have a good estimate of the depth trimmed regions. In the table, we highlight in bold the best results for each dataset. Clearly, the winner is $d_{\triangle_{\sigma},n}.$ It is worth observing that, the results obtained when selecting $\sigma$ via cross-validation on the training set, are only slightly worse than what it is reported in the best case scenario.

\section{Concluding remarks} \label{section:concluding_remarks}

We have introduced the first two generalizations of simplicial depth that do not necessarily give value zero to the elements of the space that are outside the convex hull of the distribution support. Although the literature contains instances of depth that are always positive, this is not the case for the simplicial depth that, with Tukey depth, is one of the most well-known depth functions. Tukey depth also suffers from the same problematic.
The two proposed generalizations of simplicial depth are based on the idea of enlargement. As the simplicial depth of a point is the probability of the point being in a random simplex, we can either enlarge the simplex or the probability distribution. These enlargements will depend on a constant, $\sigma,$ that we select so that the probability that the elements of interest belong to the simplex is positive, and consequently their depth values. Note that such constant always exists. Furthermore, we prove that under mild regularity conditions the behavior of the proposed generalizations varies in a continuous manner as a function of $\sigma.$

There is a complication associated with enlarging the simplex, which is the result of obtaining  a simplex with dependent vertices. This is not the case in the original simplicial depth nor when the distribution is enlarged.
By enlarging the distribution, we compute the simplicial depth with respect to the distribution associate to a linear combination of independent random variables. This leads to an interesting side problem. In relation to it, we study whether regularity conditions of the distribution are inherited by the enlarged distribution. We obtain, for instance, that continuity and spherical and central symmetry are inherited under affine combinations of independent random variables  but other symmetry notions such as halfspace and angular symmetry are not.
Among other far reaching implications, the inheritance of regularity conditions by the distribution is important to establish the properties constituting the notion of statistical data depth. From a theoretical point of view we obtain that the proposals have a good behavior.
However, there is a clear theoretical advantage of the distribution enlarged simplicial depth, as it satisfies more properties, which include maximality at the center, monotonicity and connected depth trimmed regions, under certain mild conditions.

The existing literature contains Tukey depth estimators that are not necessarily zero outside the convex hull of the sample. For the simplicial depth, we propose here the first two estimators. They are U-statistics for the  estimation of  the given generalizations of simplicial depth. We obtain uniform consistency of these estimators, consistency of their maximizers and asymptotic normality of the associated empirical processes.
These theoretical results are crucial in practical applications as they ensure that the estimators provide a good approximation of the proposed generalizations of simplicial depth. We see this in the performed real data analyses and Monte Carlo simulations.
For the simulations, we perform supervised classification on different scenarios (with the convex hull of the training and test distributions overlapping and not overlapping) showing how to select the parameter $\sigma$ appropriately. Additionally, we compare our estimators with those of the Tukey depth that are not necessarily zero outside the convex hull of the sample obtaining that our estimators have  similar or better performance.  For the real data analysis we make use of four datasets where we do a thorough analysis on the selection of $\sigma$ and compare the sample depths  with other classical methodologies for classification.
We observe that in practice, better results are obtained when making use of the simplex enlarged $\sigma$-simplicial depth estimator.

\vspace{0.5cm}
\addresses

\clearpage
\bigskip
\begin{center}
{\large\bf SUPPLEMENTARY MATERIAL}
\end{center}

\setcounter{section}{0}
\renewcommand{\thesection}{\Alph{section}}

\section{Proofs}
Unless specified otherwise, $P$ is a probability measure over the Borel sets of $\mathbb{R}^p.$

\begin{proof}[Proposition \ref{proposition:enlargement:translation}] The case of $X$ being  symmetric about $\mu$ with respect to either spherical, elliptical or central symmetry, is addressed below in the proofs of Proposition \ref{proposition:elliptical:symmetry} and \ref{proposition:spherical:central:symmetry} ($n=1$).

Let $X$ be angularly symmetric about $\mu$. According to \citet[Theorem 2.2]{Zuo-2000-b}, angular symmetry is equivalent to $\mathbb{P}(u^{\top}(X-\mu) \geq 0) = \mathbb{P}(u^{\top}(\mu-X) \geq 0)$ for all $u \in S^{p-1}$. Denote $E_u^{1} := [ u^{\top} (X-\mu) \geq 0 ]$ and $E_u^{2} := [ u^{\top} (\mu-X) \geq 0 ]$ for any $u \in S^{p-1}.$ Let $\lambda>0$. Multiplying by $\lambda$ on both sides of the inequality in $E_u^{1}$ and $E_u^{2}$ and adding and subtracting $b$ to $X,$ we have that $E_u^{1} = [ u^{\top} ( (\lambda X+b)- (\lambda \mu+b)) \geq 0 ]$ and
		$E_u^{2} = [ u^{\top} ( (\lambda \mu+b)- (\lambda X+b)) \geq 0 ].$ Due to	$\mathbb{P}(E_u^{1}) = \mathbb{P}(E_u^{2})$ and the arbitrary of $u,$  we get that $\lambda X+b$ is angularly symmetric about $\lambda \mu+b$.
The proof follows analogously for $\lambda<0,$ as we obtain 
	$E_u^{1} = [ u^{\top} ( (\lambda \mu+b)- (\lambda X+b)) \geq 0 ]$ and
	$E_u^{2} = [ u^{\top} ( (\lambda X+b)- (\lambda \mu+b)) \geq 0 ].$
 If $\lambda=0,$ $\lambda X+b$ is the degenerate random variable $b$, which, clearly, is angularly symmetric about $\lambda \mu+b = b$.

Let $X$ be halfspace symmetric about $\mu$. According to \citet[Theorem 2.4]{Zuo-2000-b},  $X$ is halfspace symmetric about $\mu$ if and only if  $\med(u^{\top}X)=u^{\top} \mu$ for all $u \in S^{p-1},$ with $\med(\cdot)$ denoting the median function. For any $\lambda \in \mathbb{R}$ and $b \in \mathbb{R}^p,$ we have that 
$
	\med(u^{\top} (\lambda X+b) ) = \lambda \med(u^{\top} X) + u^{\top} b = \lambda u^{\top} \mu + u^{\top} b = u^{\top} (\lambda \mu+b),
$ where the first equality is thanks to the median is a homogeneous function. Thus, $\lambda X+b$ is halfspace symmetric about $\lambda \mu+b$.
\end{proof}
\vspace{.5cm}

\begin{proof}[Proposition \ref{proposition:elliptical:symmetry}]
Let $X_i$ be elliptically symmetric about $\mu$   for each $i=1, \dots, n.$ As $X_1, \ldots, X_n$ are identically distributed, there exists a nonsingular matrix $V$ such that $V X_i$ is spherically symmetric about $V \mu$  for all $i=1, \dots, n.$ Thus,  $U (V X_i - V \mu) \eqd (V X_i - V \mu )$ for any orthonormal matrix $U$ and $i=1,\dots,n.$  Given $\lambda \in \mathbb{R}$ and $b \in \mathbb{R}^p,$ let us denote $Y:= \sum_{i=1}^{n} \lambda_i X_i+b$ and $\tilde{\mu}:=\sum_{i=1}^{n} \lambda_i\mu+b.$ Thus, 
\begin{equation*}
\begin{split}
	U(VY - V \tilde{\mu} ) &\eqd U \biggl( \sum_{i=1}^{n} \lambda_i (V X_i - V \mu) \biggr) \eqd \sum_{i=1}^{n} \lambda_i U (V X_i - V \mu) \\
	&\eqd \sum_{i=1}^{n} \lambda_i ( V X_i - V \mu) \eqd VY - V \tilde{\mu},
\end{split}
\end{equation*}
for any orthonormal matrix $U.$ Then, $Y$ is elliptically symmetric about $\tilde{\mu}$. 

The cases of spherical and  central symmetry are addressed below in the proof of Proposition \ref{proposition:spherical:central:symmetry}.
\end{proof}
\vspace{.5cm}

\begin{proof}[Proposition \ref{proposition:spherical:central:symmetry}]
 Given $\lambda \in \mathbb{R}$ and $b \in \mathbb{R}^p,$ let us denote  $Y:= \sum_{i=1}^{n} \lambda_i X_i+b$ and $\tilde{\mu}:=\sum_{i=1}^{n} \lambda_i\mu_i+b.$ 
Let $X_i$ be spherically symmetric about $\mu_i$ for $i=1, \dots, n.$ Then, for any orthonormal matrix $U$
	\begin{equation*}
		U(Y-\tilde{\mu}) \eqd U \biggl( \sum_{i=1}^{n} \lambda_i (X_i - \mu_i) \biggr) \eqd \sum_{i=1}^{n} \lambda_i U (X_i - \mu_i) \eqd \sum_{i=1}^{n} \lambda_i (X_i - \mu_i) \eqd Y - \tilde{\mu},
	\end{equation*}
which implies that $Y$ is spherically symmetric about  $\tilde{\mu}.$
	
According to \citet[Lemma 2.1]{Zuo-2000-b},  a random variable $X$ is centrally symmetric about $\mu \in \mathbb{R}^p$ if and only if  $u^{\top}(X-\mu) \eqd u^{\top}(\mu-X)$ for all $u \in S^{p-1}$.
Let $X_i$ be centrally symmetric about $\mu_i$ for $i=1, \dots, n.$ Then, $Y$ is centrally symmetric about  $\tilde{\mu}$ as 
	$u^{\top} ( Y - \tilde{\mu} ) \eqd \sum_{i=1}^{n} \lambda_i u^{\top} (X_i - \mu_i) \eqd \sum_{i=1}^{n} \lambda_i u^{\top} ( \mu_i - X_i) \eqd u^{\top} ( \tilde{\mu} - Y ).$
\end{proof}
\vspace{.5cm}

\begin{proof}[Corollary \ref{C4}]
The proof follows directly from Proposition \ref{proposition:spherical:central:symmetry} as elliptical symmetry implies central symmetry.
\end{proof}
\vspace{.5cm}

\begin{proof}[Proposition \ref{PSigmaCont}]
Recall that a distribution $P$ on $\mathbb{R}^p$ (or a random variable $X \sim P$) is  continuous if $P(\{ x \})=0$ for all $x \in \mathbb{R}^p$ and absolutely continuous if  $P(A)=0$ for any Lebesgue measure zero and measurable set $A \subset \mathbb{R}^p.$ Since $P_{\sigma}$ is the distribution of $\sigma X_1 + \frac{1-\sigma}{p+1} \sum_{j=1}^{p+1} X_j$ where $X_1, \dots, X_{p+1}$ are independent and identically distributed random variables with distribution $P,$ it is enough to show that a linear combination of continuous (or respectively absolutely continuous or smooth) random variables is continuous (or respectively absolutely continuous or smooth). For this, it suffices to prove that: (i)   if $X \sim P$ is continuous (or respectively absolutely continuous or smooth), then  $\lambda X$ is continuous (or respectively absolutely continuous or smooth) for any $\lambda \in \mathbb{R} \setminus \{ 0 \};$ and (ii) if $X \sim P$ and $Y \sim Q$ are continuous (or respectively absolutely continuous or smooth), then  the sum $X+Y$ is continuous (or respectively absolutely continuous or smooth). 

For (i), observe that $\lambda X \sim P^{\lambda}$ where  $P^{\lambda}(B):=P(\frac{1}{\lambda} B)$  for all measurable subsets $B \subset \mathbb{R}^p,$ while for (ii), notice that $X+Y$ has distribution $P*Q$ such that 
$	(P * Q) (B) = \int P(B + \{ y \} ) \d Q(y)$
for $B \subset \mathbb{R}^p$ measurable.
In both cases, the result follows by considering sets $B$ of a specific form:  if $P$ is continuous take $B=\{ x \}$ for $x \in \mathbb{R}^p$, if $P$ is absolutely continuous consider Lebesgue measure $0$ sets and  if $P$ is smooth consider hyperplanes in $\mathbb{R}^p$.
\end{proof}
\vspace{.5cm}

\begin{proof}[Proposition \ref{closure:F:sigma:under:symmetry}]
	It is a direct consequence of Proposition \ref{proposition:elliptical:symmetry}.
\end{proof}
\vspace{.5cm}

The next lemma shows that simplices are nested, which is required for the proof of Propositions \ref{continuity:of:sigma:depth:dependent:vertices:wrt:sigma} and \ref{proposition:sample:sigma:depth:pathwise:nondecreasing}.

\begin{lemma} \label{lemma:simpleces:are:increasing:with:sigma}
For all $x_1, \dots, x_{p+1} \in \mathbb{R}^p$ and scalars $\sigma^* \geq \sigma > 0,$ we have that $\triangle_{\sigma^{*}}[x_{1},\ldots,x_{p+1}]) \supset \triangle_\sigma[x_{1},\ldots,x_{p+1}])$.
\end{lemma}
\begin{proof}[Lemma \ref{lemma:simpleces:are:increasing:with:sigma} ]
Let $y \in \triangle_\sigma[x_{1},\ldots,x_{p+1}]$ and notice that $\triangle_\sigma[x_{1},\ldots,x_{p+1}]= \triangle[y_{1},\ldots,y_{p+1}]$ for
\begin{equation} \label{y}
y_i = \sigma (x_i - \bar{x}) + \bar{x},  \mbox{ with } i=1, \ldots, p+1,
\end{equation}
where $\bar{x} = \sum_{j=1}^{p+1} x_j/(p+1)$. Due to \citet[Equation (1.8)]{Liu-1990}, there exist  $\alpha_1, \dots, \alpha_{p+1} \geq 0$ with $\alpha_1+\dots+\alpha_{p+1}=1$ such that $y = \alpha_1 y_1 + \dots + \alpha_{p+1} y_{p+1};$ which by (\ref{y}) results in
$$y = \sigma \biggl( \sum_{i=1}^{p+1} \alpha_i x_i - \bar{x} \biggr) + \bar{x}.$$ Multiplying and diving by $\sigma^{*}$ the first term of the sum,
\begin{equation} \label{y:sigma:*}
\begin{split}
y &= \sigma^{*} \biggl( \sum_{i=1}^{p+1} \biggl( \frac{\sigma}{\sigma^{*}} \alpha_i \biggr) x_i - \biggl( \frac{\sigma}{\sigma^{*}} \biggr) \bar{x} \biggr) + \bar{x} \\
&= \sigma^{*} \biggl( \sum_{i=1}^{p+1} \alpha^{*}_i x_i - \bar{x} \biggr) + \bar{x}
\end{split}
\end{equation}
with $\alpha^{*}_i := \frac{\sigma}{\sigma^{*}} \alpha_i + \frac{\sigma^{*}-\sigma}{\sigma^*}$ for $i=1, \dots, p+1.$  Observe that $\alpha^{*}_i \geq 0$ and $\sum_{i=1}^{p+1} \alpha^{*}_i =1$. We rewrite equation \eqref{y:sigma:*} as $y = \alpha^{*}_1 y^{*}_1 + \dots + \alpha^{*}_{p+1} y^{*}_{p+1}$ where $y^{*}_i = \sigma^{*} (x_i - \bar{x}) + \bar{x}$ for $ i=1,2,\ldots,p+1.$ Thus, $y \in\triangle[y^{*}_{1},\ldots,y^{*}_{p+1}].$ The proof finishes because $\triangle_{\sigma^{*}}[x_{1},\ldots,x_{p+1}]=\triangle[y^{*}_{1},\ldots,y^{*}_{p+1}]$ by the definition of enlarged simplex.
\end{proof}
\vspace{.5cm}

\begin{proof}[Proposition \ref{continuity:of:sigma:depth:dependent:vertices:wrt:sigma}]
The simplex enlarged simplicial depth is monotonically non-decreasing as a function of $\sigma$ as a direct consequence of above Lemma \ref{lemma:simpleces:are:increasing:with:sigma}. We prove next the continuity on the right and the continuity for smooth $P$ of the $\sigma$-simplicial depth. Let $\sigma> 0$ and  $\{ \sigma_n \}_{n=1}^{\infty}$  be a sequence of real numbers that converges to $\sigma.$ For any $n \in \mathbb{N},$  $x\in\mathbb{R}$ and $P$ on $\mathbb{R}^p,$ we have
$$\abs{ d_{\triangle_{\sigma_n}}(x; P) - d_{\triangle_{\sigma}}(x; P) } \leq \int \abs{ g_n(x_1,\dots,x_{p+1}) - g(x_1,\dots,x_{p+1})} \d P(x_1) \dots \d P(x_{p+1}), 
$$ 
where $g(x_1,\dots,x_{p+1}):= \mathbf{I} ( x \in \triangle_{\sigma}[x_1,\dots,x_{p+1}])$ and analogously for $g_n,$ $n\geq1.$ 
Clearly $\{ g_n \}_{n=1}^{\infty}$ is measurable and bounded by $1$. 
Moreover, for $(x_1,\dots,x_{p+1})$ fixed we have two possibilities: (a) if $x$ is not on the boundary of $\triangle_{\sigma}[x_1, \dots, x_{p+1}],$ there exists an $\epsilon>0$ and $N \in \mathbb{N}$ such that $\abs{\sigma-\sigma_n} < \epsilon$ and $g_n(x_1,\dots,x_{p+1})=g(x_1,\dots,x_{p+1})$ for all $n \geq N;$ (b) if $x$ is on the boundary of $\triangle_{\sigma}[x_1, \dots, x_{p+1}]$, 
$
	\lim_{\sigma_n \uparrow \sigma} \mathbf{I} ( x \in \triangle_{\sigma_n}[x_1, \dots, x_{p+1}]) = 0
$
and $
	\lim_{\sigma_n \downarrow \sigma} \mathbf{I} ( x \in \triangle_{\sigma_n}[x_1, \dots, x_{p+1}]) = 1.
$ 
Therefore, if $\{ \sigma_n \}_{n=1}^{\infty}$  converges from above to $\sigma,$ the corresponding sequence of functions $\{ g_n \}_{n=1}^{\infty}$ convergence pointwise to $g$ and because of dominated convergence theorem, the distribution enlarged $\sigma$-simplicial depth is right continuous.
If $P$ is smooth the set $\{ (x_1,\dots,x_{p+1}) \in \mathbb{R}^{p} \times \dots \times \mathbb{R}^{p} \, : \,  x \in \partial \triangle_{\sigma}[x_1, \dots, x_{p+1}] \}$ has measure $0$. Hence, $h_n$ converges pointwise to $h$ almost everywhere. The result follows again from dominated convergence theorem.

The proof of the continuity of the distribution enlarged $\sigma$-simplicial depth for smooth $P$ is similar. Each $P_\sigma,$ and $P_{\sigma_n},$ is replaced  by $p+1$ $P$'s, so that the dependence on $\sigma$ is not in the probability distribution but only in the integrand (see Definition \ref{Psigma}).
For any $n \in \mathbb{N},$  $x\in\mathbb{R}$ and $P$ on $\mathbb{R}^p,$
\begin{equation*}
  \abs{ d_{P_{\sigma_n}}(x; P) - d_{P_{\sigma}}(x; P) } \leq \int \abs{ h_n(x_1,\dots,x_{(p+1)^2}) - h(x_1,\dots,x_{(p+1)^2})} \d P(x_1) \dots \d P(x_{(p+1)^2})
\end{equation*}
where $h(x_1,\dots,x_{(p+1)^2}) := \mathbf{I} ( x \in \triangle[\sigma (x_1-\bar{x}_{1})+\bar{x}_{1}, \ldots, \sigma (x_{1+p(p+1)} - \bar{x}_{p+1}) + \bar{x}_{p+1} ] ),$ $\bar{x}_i := \sum_{j=1}^{p+1} x_{j+(i-1)(p+1)}/(p+1),$ and analogously for $h_n,$ $n\geq1.$ The result follows again from dominated convergence theorem and the almost sure pointwise convergence of the sequence $\{ h_n \}_{n=1}^{\infty}$ to $h$.

We prove next the monotonicity of the distribution enlarged sigma simplicial depth. If $P$ has density $f(x)=\kappa g((x-\mu)^{\top} \Sigma^{-1} (x-\mu))$ from \citet[Theorem 3.3]{Zuo-2000-c} we see that $d_{P_{\sigma}}(x; P) = h((x-\mu)^{\top} \Sigma_{\sigma}^{-1} (x-\mu)),$ where $h$ is a nonincreasing function and $\Sigma_{\sigma}=\left( \sigma^2 + \frac{1-\sigma^2}{p+1} \right) \Sigma$. Therefore, for any scalars $\sigma^{*} \geq \sigma >0,$ $d_{P_{\sigma^{*}}}(x; P) \geq d_{P_{\sigma}}(x; P)$. 
\end{proof}
\vspace{.5cm}

\begin{proof}[Theorem \ref{PropertiesS1}]
(i) The affine invariance property follows from \citet[Equation (1.8)]{Liu-1990} and the fact that, for any $p \times p$ nonsingular matrix $A$ and $p$ dimensional vector $b$, $A y + b \in \triangle[ A y_1+b,\dots,A y_{p+1}+b]$ if and only if $y \in \triangle[y_1,\dots,y_{p+1}]$. (ii) For the vanishing at infinity property, observe that
\begin{equation*}
\begin{split}
	\lim_{\norm{x} \rightarrow \infty} d_{\triangle_{\sigma}} (x; P) & \leq \limsup_{\norm{x} \rightarrow \infty} d_{\triangle_{\sigma}} (x; P) \\
	& \leq \int \limsup_{ \norm{x} \rightarrow \infty } \mathbf{I}( x\in\triangle_\sigma[x_{1},\dots,x_{p+1}]) \, \d P(x_1) \dots \d P(x_{p+1}) = 0,
\end{split}
\end{equation*}
because of dominated convergence theorem. (iii) For the upper semicontinuity property,  let $x^{*} \in \mathbb{R}^p$ and note that
\begin{equation*}
\begin{split}
\limsup_{x \rightarrow x^{*} } d_{\triangle_{\sigma}}(x; P)
& \leq \int \limsup_{x \rightarrow x^{*} } \mathbf{I}( x\in\triangle_\sigma[x_{1},\dots,x_{p+1}]) \, \d P(x_1) \dots \d P(x_{p+1}) \\
& \leq \int \mathbf{I}( x^{*} \in\triangle_\sigma[x_{1},\dots,x_{p+1}]) \, \d P(x_1) \dots \d P(x_{p+1}) \\
&= d_{\triangle_{\sigma}}(x^{*}; P).
\end{split}
\end{equation*}
(iv) Finally, let $P$ be smooth, then
\begin{equation*}
\begin{split}
\liminf_{x \rightarrow x^{*} } d_{\triangle_{\sigma}}(x; P) & \geq 
\int \liminf_{x \rightarrow x^{*} } \mathbf{I}( x\in\triangle_\sigma[x_{1},\dots,x_{p+1}]) \, \d P(x_1) \dots \d P(x_{p+1}) \\
&= d_{\triangle_{\sigma}}(x^{*}; P).
\end{split}
\end{equation*}
\end{proof}
\vspace{.5cm}

\begin{proof}[Proposition \ref{proposition:monotonicityoutsidesupport}]
For any $x\in\mathbb{R}$ and $P$ on $\mathbb{R},$ 
\begin{equation*}
\begin{split}
d_{\triangle_\sigma}(x;P)&=\\
=& \int \mathbf{I} \left( \frac{1+\sigma}{2} x_1 + \frac{1-\sigma}{2} x_2 \leq x \leq \frac{1-\sigma}{2} x_1 + \frac{1+\sigma}{2} x_2 \right) \, \d P(x_1) \d P(x_2) \\
+& \int \mathbf{I} \left( \frac{1-\sigma}{2} x_1 + \frac{1+\sigma}{2} x_2 \leq x \leq \frac{1+\sigma}{2} x_1 + \frac{1-\sigma}{2} x_2 \right) \, \d P(x_1) \d P(x_2) \\
-& \int \mathbf{I} \left( x_1=x_2=x \right) \, \d P(x_1) \d P(x_2).
\end{split}
\end{equation*}
For any $x\in\mathbb{R}$, let us denote
\begin{equation*}
\begin{split}
S_1^{+}(x) &:= \left \{ (x_1, x_2) \in \mathbb{R}^2 \, : \, \frac{1+\sigma}{2} x_1 + \frac{1-\sigma}{2} x_2 \le x \right \}, \\
S_2^{+}(x) &:= \left \{ (x_1, x_2) \in \mathbb{R}^2 \, : \, \frac{1-\sigma}{2} x_1 + \frac{1+\sigma}{2} x_2 \ge x \right \}, \\
S_1^{-}(x) &:= \left \{ (x_1, x_2) \in \mathbb{R}^2 \, : \, \frac{1+\sigma}{2} x_1 + \frac{1-\sigma}{2} x_2 \ge x \right \},\\
S_2^{-}(x) &:= \left \{ (x_1, x_2) \in \mathbb{R}^2 \, : \, \frac{1-\sigma}{2} x_1 + \frac{1+\sigma}{2} x_2 \le x \right \},\\
\end{split}
\end{equation*}
$S^+(x):=S_1^{+}(x) \cap S_2^{+}(x) \mbox{ and }
S^-(x):=S_1^{-}(x) \cap S_2^{-}(x).$
Then, 
\begin{equation*}
d_{\triangle_\sigma}(x;P)=\int_{S^{+}(x)} \, \d P(x_1) \d P(x_2) + \int_{S^{-}(x)} \, \d P(x_1) \d P(x_2) - \int _{\{(x,x)\}} \, \d P(x_1) \d P(x_2).
\end{equation*}
We consider the case $b \leq y \leq x.$ The proof for the case $x \le y \leq a$ is analogous. Thus, 
\begin{equation*}
\begin{split}
d_{\triangle_\sigma}(y;P) -& d_{\triangle_\sigma}(x;P) =\\
&= \int_{S^{+}(y)\setminus S_{2}^{+}(x)} \, \d P(x_1) \d P(x_2) + \int_{S^{-}(y)\setminus S_{1}^{-}(x)} \, \d P(x_1) \d P(x_2)\\
 &- \int _{\{(y,y)\}} \, \d P(x_1) \d P(x_2) - \int_{S^{+}(x) \setminus S_{1}^{+}(y)} \, \d P(x_1) \d P(x_2) \\
 &- \int_{S^{-}(x) \setminus S_{2}^{-}(y)} \, \d P(x_1) \d P(x_2)  +  \int _{\{(x,x)\}} \, \d P(x_1) \d P(x_2).
\end{split}
\end{equation*}
Notice that $S^{+}(x)  \setminus S_{1}^{+}(y)$ and $S^{-}(x)  \setminus S_{2}^{-}(y)$ have no intersection with $S \times S$. Hence, their corresponding probability is zero. Furthermore $\{(y,y)\}$ is a subset of both $S^{+}(y) \setminus S_{2}^{+}(x)$ and $S^{-}(y)  \setminus S_{1}^{-}(x)$. All this implies that $d_{\triangle_\sigma}(y;P) - d_{\triangle_\sigma}(x;P) \ge 0$.
\end{proof}
\vspace{.5cm}

\begin{proof}[Theorem \ref{PropertiesS2}]
As the distribution enlarged $\sigma$-simplicial depth with respect to a distribution $P$ is the simplicial depth with respect to the corresponding distribution $P_{\sigma},$  properties (i) and (ii) follow from \citet[Equations (1.8) and (1.9), Theorem 1]{Liu-1990} and the, above, proof of Theorem \ref{PropertiesS1}. 
The proofs of properties (iii) and (iv) are similar to that of Theorem \ref{PropertiesS1} (see also \citet[Theorem 2]{Liu-1990}). Observe that thanks to Theorem \ref{PSigmaCont}, if $P$ is smooth, then  $P_\sigma$ is smooth.  
Finally, if $P$ is centrally symmetric then, because of Theorem  \ref{closure:F:sigma:under:symmetry}, $P_{\sigma}$ is  centrally symmetric about the same point. Therefore, properties (v) and (vi) follow from \citet[Theorem 3]{Liu-1990}.
\end{proof}
\vspace{.5cm}

\begin{proof}[Proposition \ref{proposition:depthtrimmedregionsigmadepth}]
(i) is a consequence of the affine invariance of the $\sigma$-simplicial depths (see (i) of Theorem \ref{PropertiesS1} and Theorem \ref{PropertiesS2}, and (a) of \citet[Theorem 3.1]{Zuo-2000-c}). (ii) follows directly from (b) of \citet[Theorem 3.1]{Zuo-2000-c}. For the distribution enlarged $\sigma$-simplicial depth,  (iii) holds because point (iii) of Theorem \ref{PropertiesS2} implies that the depth trimmed regions are closed and point (ii) of Theorem \ref{PropertiesS2} implies that the depth trimmed regions are bounded. For the simplex enlarged $\sigma$-simplicial depth, the proof is the same using  Theorem \ref{PropertiesS1}.
(iv) is a consequence of point (vi) of Theorem \ref{PropertiesS2} and point (c) of \citet[Theorem 3.1]{Zuo-2000-c}.
\end{proof}
\vspace{.5cm}

\begin{proof}[Proposition
\ref{proposition:sample:sigma:depth:pathwise:nondecreasing}]
$d_{\triangle_\sigma,n}(x;P)$ is an average of indicators of the form $\mathbf{I}(x \in \triangle_{\sigma}[x_1, \dots, x_{p+1}])$ for $x_1, \dots, x_{p+1} \in \mathbb{R}^p$. Then, the result follows from Lemma \ref{lemma:simpleces:are:increasing:with:sigma}. 
\end{proof}
\vspace{.5cm}

\begin{proof}[Theorem \ref{uniform:consistency}]
For any distribution $P$ on $\mathbb{R}^{p},$ $d_{\triangle_{\sigma},n}(\cdot; P)$ and $d_{P_{\sigma}, k}(\cdot; P)$ are clearly $U$-statistic of order $p+1$ with symmetric kernels for the estimation of $d_{\triangle_\sigma}(\cdot; P)$ and $d_{P_{\sigma}}(\cdot; P)$, respectively. Moreover the class of functions indexing each of them are collections of indicators of a VC-class of sets (i.e.\ simplices in $\mathbb{R}^{p}$). 
For the almost sure uniform convergence of $d_{\triangle_\sigma,n}(\cdot,P)$ to $d_{\triangle_\sigma}(\cdot,P),$ observe that the only difference with the classical simplicial depth is the rescaling of the simplices. Therefore,  \citet[Corollary 6.7]{Arcones-1993} holds. The result follows from Corollary 3.3 therein. 
Due to $d_{P_{\sigma},k(n,p)}(\cdot,P)=d_{S,k(n,p)}(\cdot,P_{\sigma})$ is the classical sample simplicial depth based on $k$ independent random draws with distribution $P_{\sigma},$ its almost sure uniform convergence to $d_{P_{\sigma}}(\cdot,P)$ follows from  \citet[Corollary 6.8]{Arcones-1993}. 
\end{proof}
\vspace{.5cm}

\begin{proof}[Corollary \ref{convergence:of:the:sample:median}]
The $\sigma$-simplicial depths $d_{\triangle_\sigma}(\cdot,P)$ and $d_{P_{\sigma}}(\cdot,P)$ are upper semicontinuous and vanish at infinity because of Theorem \ref{PropertiesS1} and \ref{PropertiesS2}, respectively. Furthermore, Theorem \ref{uniform:consistency} implies that, for either of the two cases of $(d,d_n)$, $d_n(\cdot;P)$ converges uniformly almost surely to $d(\cdot;P)$. Therefore, the proof of the almost sure convergence of $\mu_n$ to $\mu$ is analogous to that of \citet[Theorem 6.9]{Arcones-1993}.
\end{proof}
\vspace{.5cm}

\begin{proof}[Theorem \ref{CLT}]
It follows from \citet[Theorem 4.9]{Arcones-1993}. In particular, $g_x(z) = \int k_x(x_1, \dots, x_p, z) \, \d Q(x_1) \dots \d Q(x_p),$ where $k_x$ is the kernel of the corresponding U-statistic for (i) and (ii).
\end{proof}

\section{Further simulations and results}\label{Asim}

As commented in Section \ref{Simulations},  we represent in Figure \ref{boxplots:sigma:simplicial:version:1:out}  the misclassification rates of the outsiders for the sample distribution enlarged $\sigma$-simplicial depth, $d_{P_\sigma,k},$ whose performance is slightly worse than that of the simplex enlarged. 
\begin{figure}[!h]
	\includegraphics[width=\linewidth]{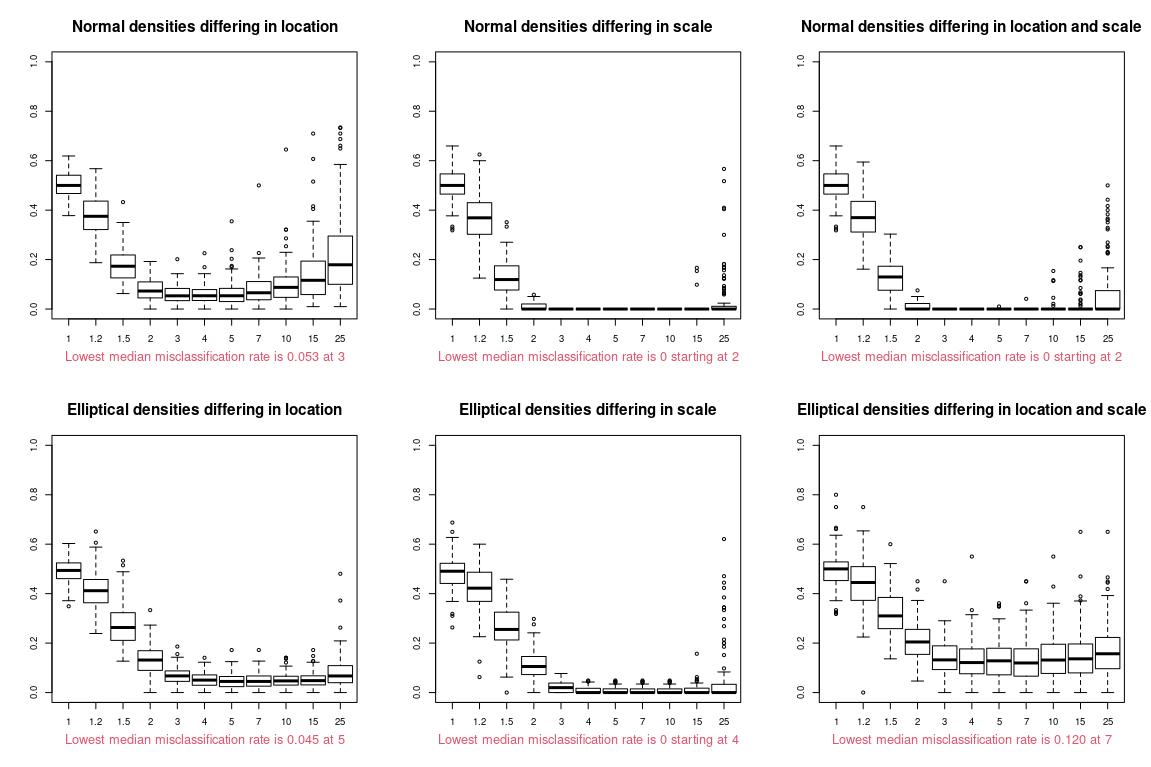}
	\caption{Boxplots of  $100$ misclassification rates of the outsiders for $d_{P_\sigma,k}$ with $\sigma\in\{1.2, 1.5, 2, 3, 4, 5, 7, 10, 15, 25\}$ and the sample simplicial depth ($\sigma=1$). Linear DD-plot classifier.}
 	\label{boxplots:sigma:simplicial:version:1:out}
\end{figure}

In this section we also report the misclassification rates for the whole sample and  all depth functions. For the whole sample, the $\sigma$-simplicial depths and the refined halfspace depth perform well. On the contrary, the illumination depth shows a bad performance when the difference is in scale. See Figures \ref{TtriangleT}, \ref{figure:refined:halfspace}, \ref{figure:illumination_depth}, and \ref{BskT}. 
\begin{figure}[!h]
\centering
\includegraphics[width=\linewidth,height=.5\linewidth]{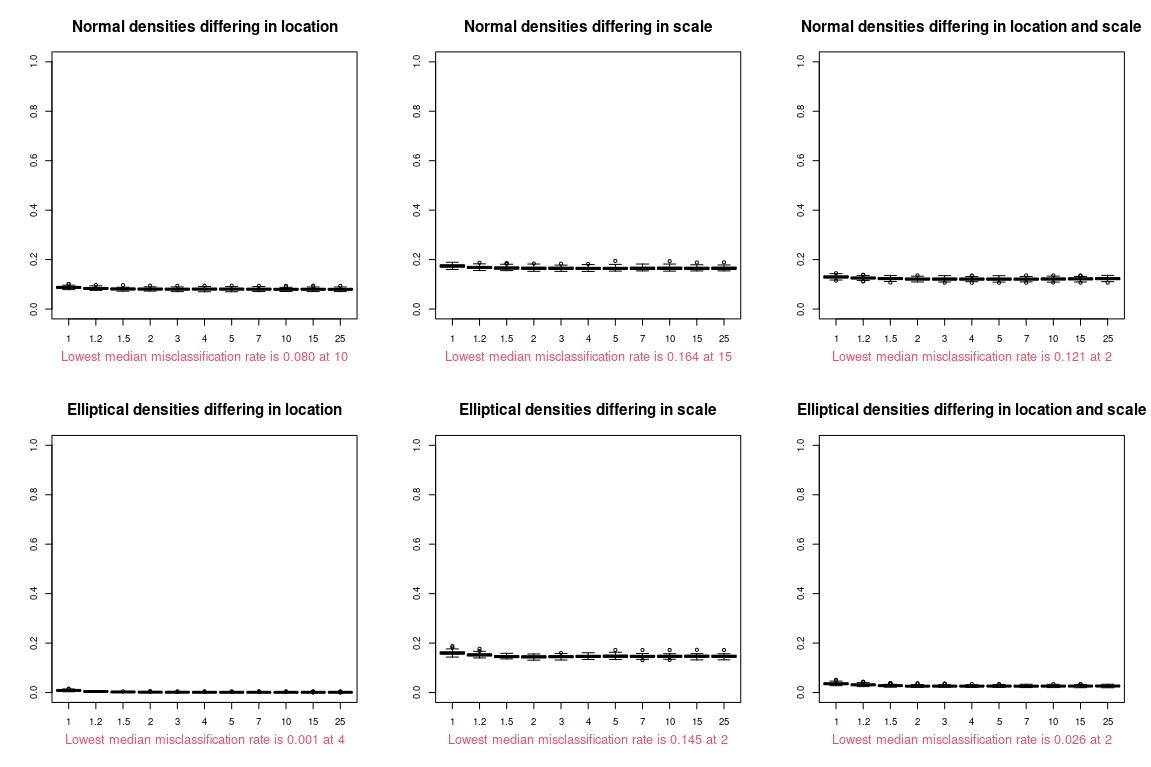}
\caption{\label{TtriangleT} Boxplots of $100$ misclassification rates of the whole sample for $d_{\triangle_{\sigma},n}$ with $\sigma\in\{1.2, 1.5, 2, 3, 4, 5, 7, 10, 15, 25 \}$ and the sample simplicial depth ($\sigma=1$). Linear DD-plot classifier.
}
\end{figure}
Considering the performance for the outsiders and  the whole sample in all scenarios, we suggest the use of $d_{\triangle_{\sigma},n}.$ $d_{P_\sigma,k}$ is a good competitor for large sample sizes, although in this case the choice of $\sigma$ is more important. 

\begin{figure}[!h]
  \centering
  \includegraphics[width=\linewidth,height=.5\linewidth]{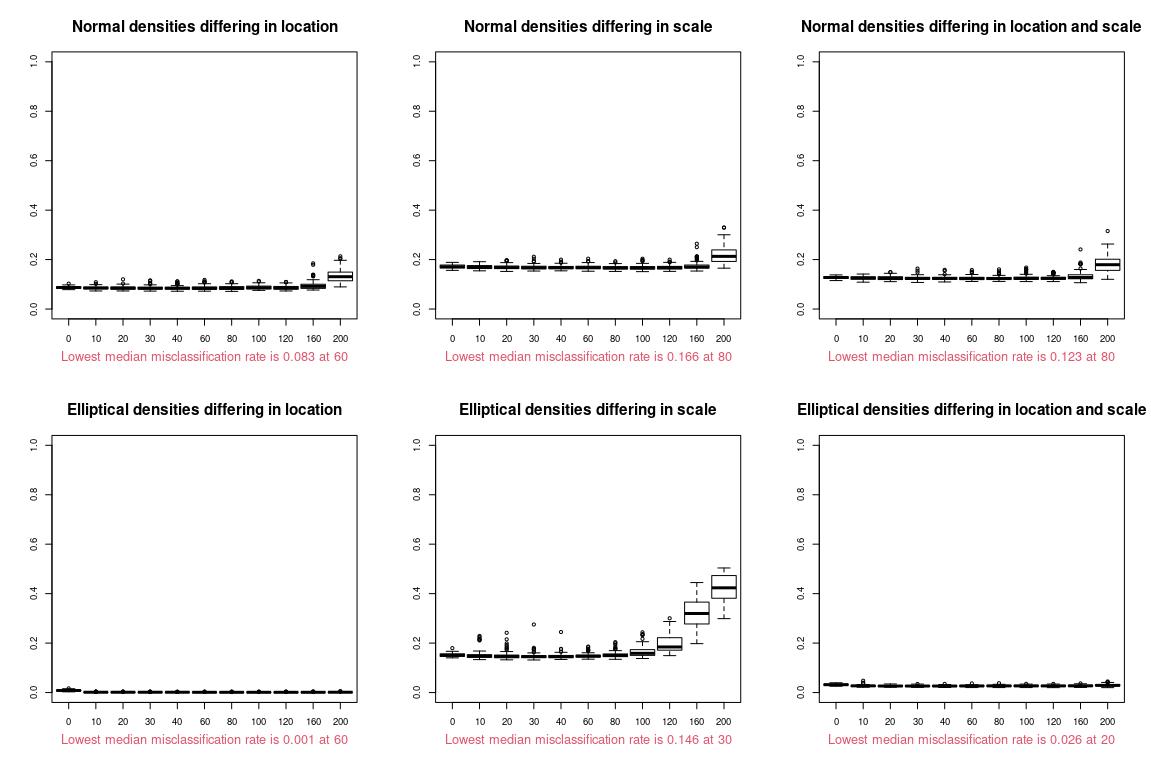}
  \caption{\label{figure:refined:halfspace} Boxplots of $100$ misclassification rates of the whole sample for the refined halfspace depth with $r\in\{10, 20, 30, 40, 60, 80, 100, 120, 160, 200\}$ and the sample halfspace depth ($r=0$). Linear DD-plot classifier.}
\end{figure}

 Furthermore, we compute the misclassification rates using $k$-NN, LDA and QDA for both the whole sample and the outsiders. Figures \ref{figure:knn_all} and \ref{figure:knn_out} show the corresponding boxplots of misclassification rates for $k$-NN across different values of $k,$ from one to ten. The median misclassification rates for LDA and QDA are reported in Table \ref{table:misclassification_rates_whole_sample} for the whole sample and Table \ref{table:misclassification_rates_outsiders} for the outsiders. These tables also contain the misclassification  rates for the other methodologies; in particular, the lowest misclassification rates across the parameter values. The best results for each distribution are highlighted in bold. Since the distributions considered are elliptically symmetric, QDA performs very well and it can be considered as a benchmark. 
 
 We consider an additional classification setting, involving random variables distributed around the unit circle. Specifically, in the framework of {\it Simulation 1}, we take $X_{1}^{(1)} \stackrel{d}{=} Z_{1}/\sqrt{2}$, where $Z_{1}$ is the standard normal bivariate distribution, and $X_{1}^{(2)} \stackrel{d}{=} (1+Y_{1}) \cdot (\cos(U_{1}), \sin(U_{1}))^{\top}$, where $Y_{1}$ is exponentially distributed with mean $10^{-4}$ and $U_{1}$ is uniformly distributed over the interval $(0,2\pi)$ and independent of $Y_{1}$. We report the median misclassification rates for the whole sample and the outsiders in the last column of Tables \ref{table:misclassification_rates_whole_sample} and \ref{table:misclassification_rates_outsiders}, respectively. In this setting, LDA and QDA perform poorly. The lowest misclassification rates for whole samples are obtained using $k$-NN, followed by $d_{\triangle_{\sigma},n}$, $d_{P_\sigma,k}$, and refined halfspace depth. $d_{\triangle_{\sigma},n}$, refined halfspace depth, illumination depth and $k$-NN all perfectly classify the outsiders.

\begin{figure}
  \centering
  \includegraphics[width=\linewidth,height=.5\linewidth]{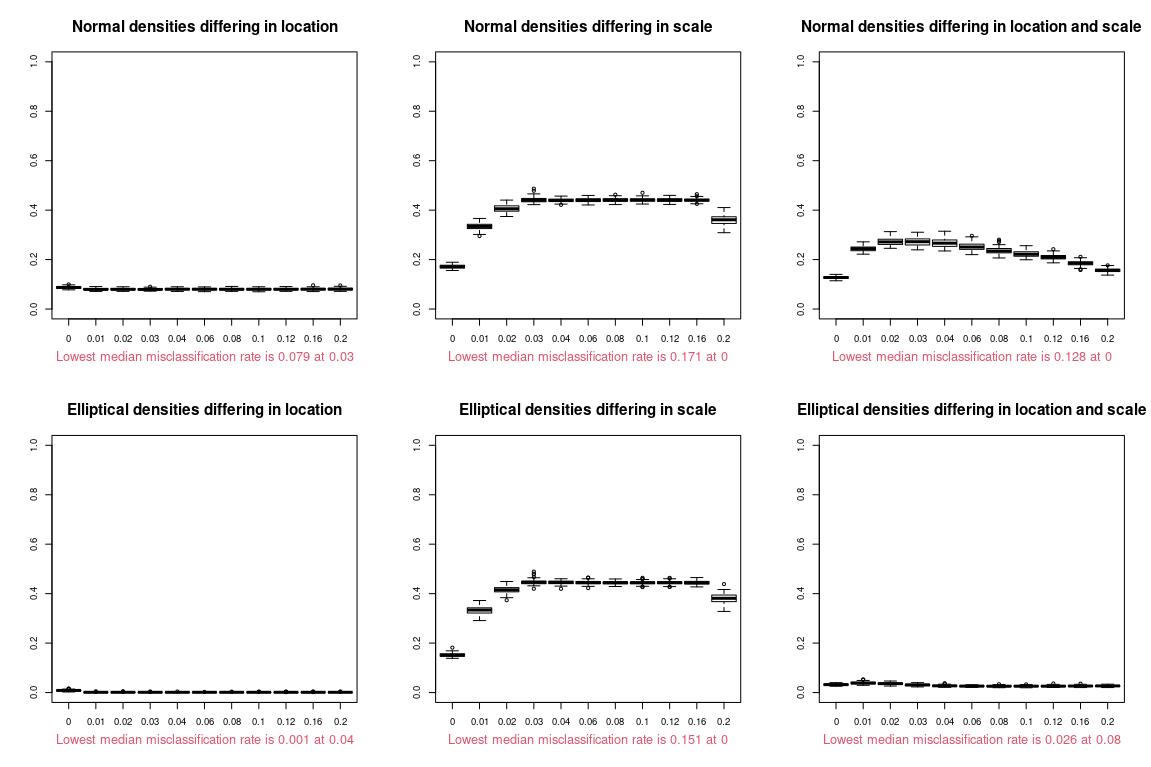}
\caption{\label{figure:illumination_depth} Boxplots of $100$ misclassification rates of the whole sample for the sample illumination depth with $\alpha \in \{0.01, 0.02, 0.03, 0.04, 0.06, 0.08, 0.10, 0.12, 0.16, 0.20 \}$ and the sample halfspace depth ($\alpha=0$). Linear DD-plot classifier.}
\end{figure}

\begin{figure}
  \centering
\includegraphics[width=\linewidth,height=.5\linewidth]{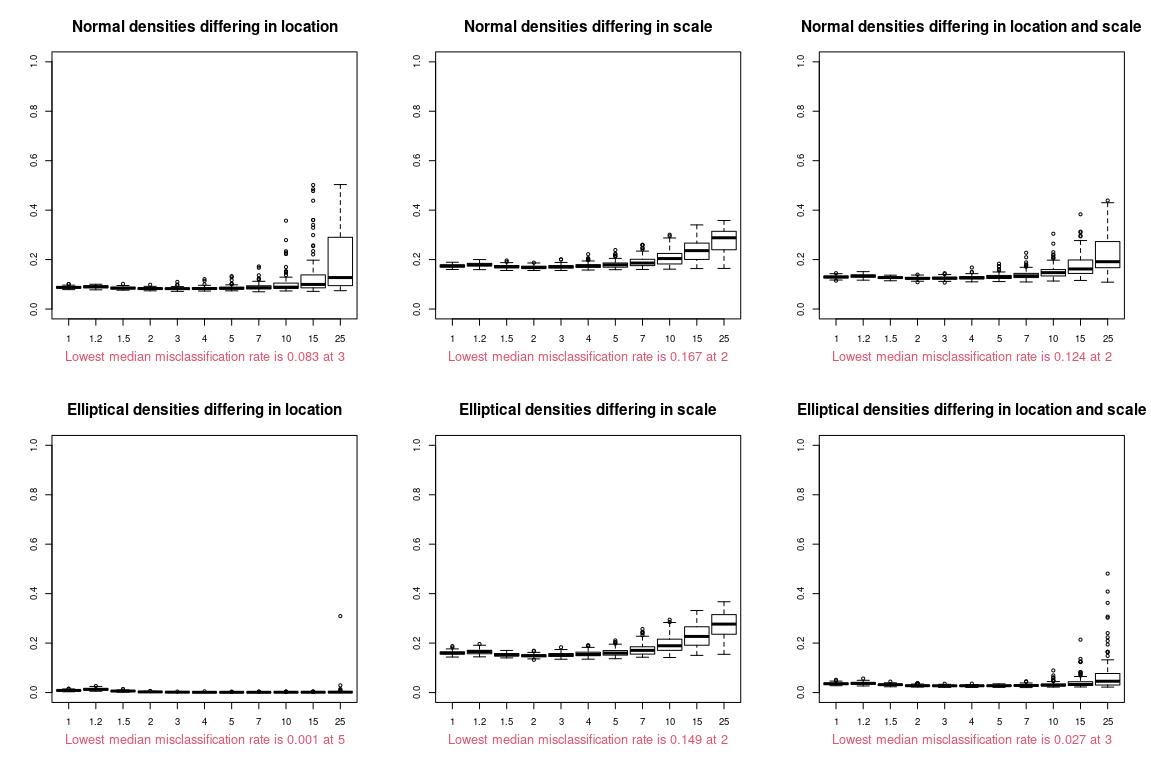}
\caption{\label{BskT} Boxplots of $100$ misclassification rates of the whole sample for $d_{P_\sigma,k}(\cdot,\cdot)$ with $\sigma\in\{1.2, 1.5, 2, 3, 4, 5, 7, 10, 15, 25\}$ and the sample simplicial depth ($\sigma=1$). Linear DD-plot classifier.}
\end{figure}

\begin{figure}
  \centering
  \includegraphics[width=\linewidth,height=.5\linewidth]{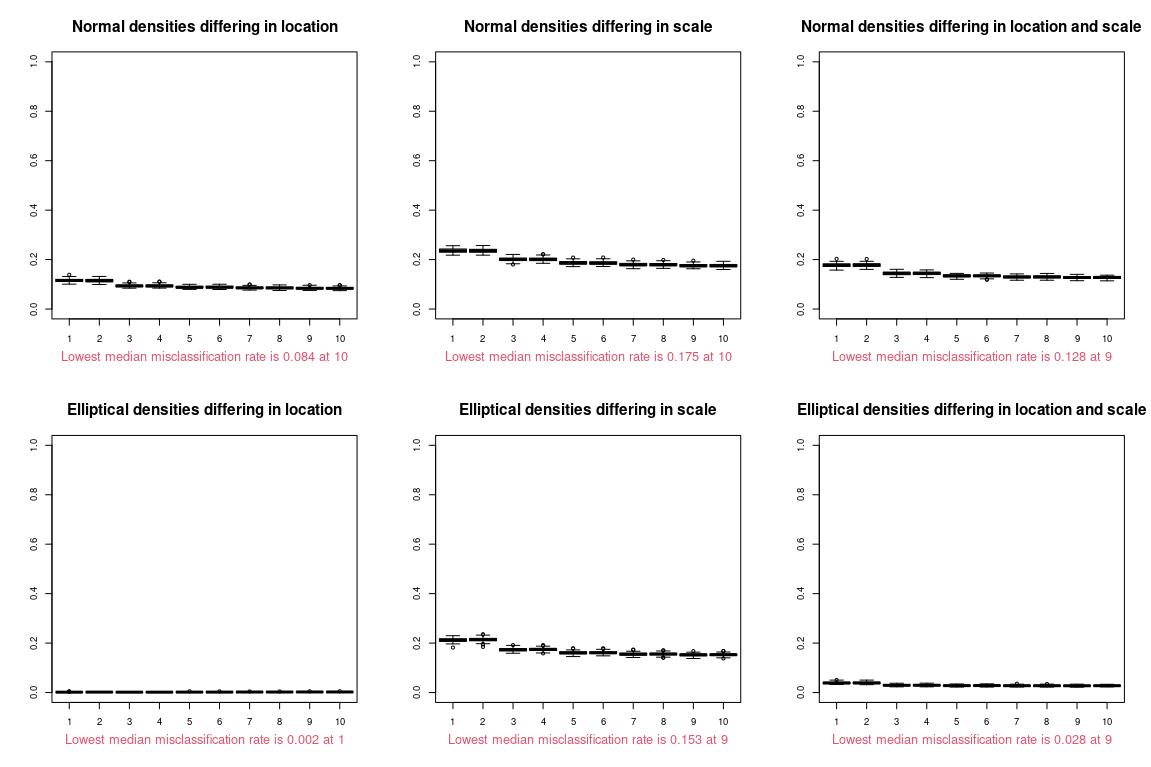}
\caption{\label{figure:knn_all} Boxplots of $100$ misclassification rates of the whole sample for $k$-NN with $k \in \{1, 2, 3, 4, 5, 6, 7, 8, 9, 10 \}$.}
\end{figure}

\begin{figure}
  \centering
  \includegraphics[width=\linewidth,height=.5\linewidth]{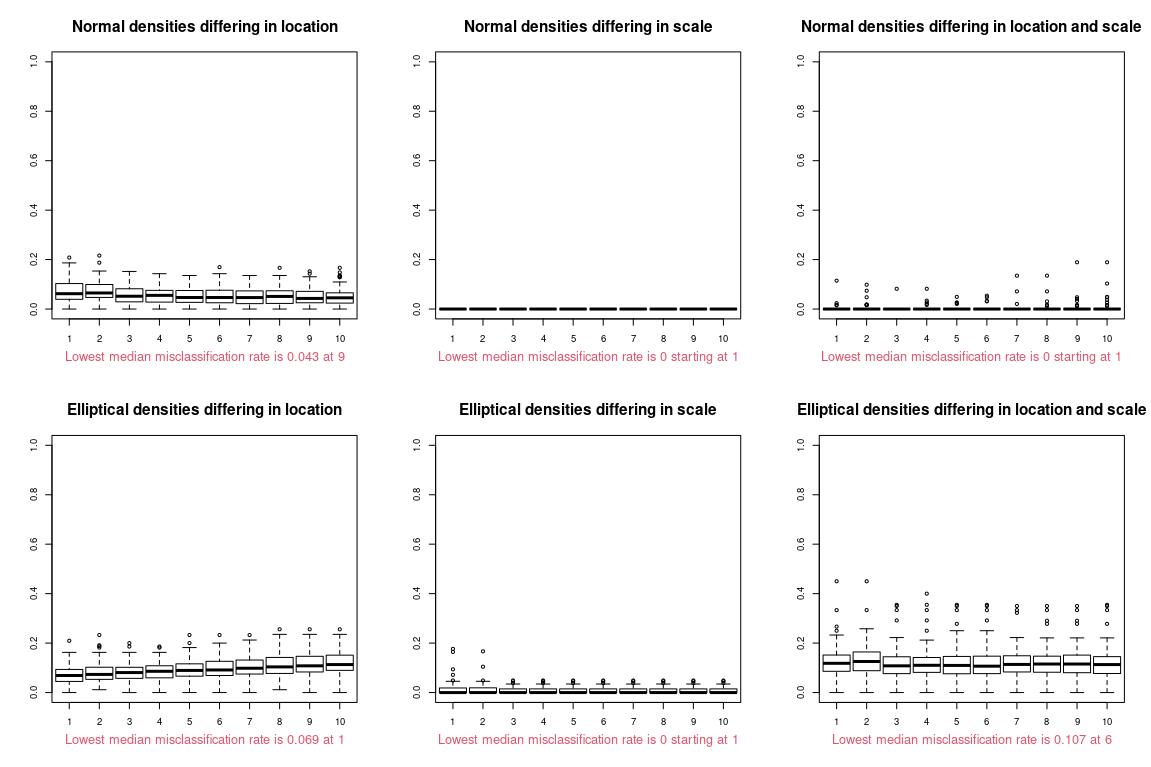}
\caption{\label{figure:knn_out} Boxplots of $100$ misclassification rates of the outsiders for $k$-NN with $k \in \{1, 2, 3, 4, 5, 6, 7, 8, 9, 10 \}$.}
\end{figure}

\begin{table}
  \caption{\label{table:misclassification_rates_whole_sample} Lowest median misclassification rate, across parameters, of the whole sample for $d_{\triangle_{\sigma},n}$, refined halfspace depth, illumination depth, $d_{P_\sigma,k}$, and $k$-NN. In the last two rows, the median misclassification rate of the whole sample for LDA and QDA, respectively.}
\vspace{0.5cm}  
 \centering
\small\addtolength{\tabcolsep}{-3pt}
\begin{tabular}{ p{2.2cm} p{1.5cm} p{1.5cm} p{1.5cm}p{.6cm}  p{1.5cm} p{1.5cm} p{1.5cm} p{.6cm} p{1.5cm} }  
& \multicolumn{3}{c}{ Bivariate normal} &&  \multicolumn{3}{c}{ Bivariate elliptical} & & Circular \\
& location & scale & location \& scale && location & scale & location \& scale & &\\
\hline
$d_{\triangle_{\sigma},n}$ & 0.080 & 0.164 & 0.121 && {\bf{0.001}} & 0.145 & 0.026 && 0.193 \\
ref.\ halfspace & 0.083 & 0.166 & 0.123 && {\bf{0.001}} & 0.146 & 0.026 && 0.343 \\
illumination & 0.079 & 0.171 & 0.128 && {\bf{0.001}} & 0.151 & 0.026 && 0.465 \\
$d_{P_\sigma,k}$ & 0.083 & 0.167 & 0.124 && {\bf{0.001}} & 0.149 & 0.027 & &0.284 \\
$k$-NN & 0.084 & 0.175 & 0.128 && 0.002 & 0.153 & 0.028 && {\bf{0.054}} \\
LDA & {\bf{0.078}} & 0.489 & 0.200 && {\bf{0.001}} & 0.491 & 0.056 && 0.498 \\
QDA & {\bf{0.078}} & {\bf{0.161}} & {\bf{0.118}} && {\bf{0.001}} & {\bf{0.144}} & {\bf{0.025}} && 0.504
\end{tabular}
\end{table}

\begin{table}
  \caption{\label{table:misclassification_rates_outsiders} Lowest median misclassification rate, across parameters, of the outsiders for $d_{\triangle_{\sigma},n}$, refined halfspace depth, illumination depth, $d_{P_\sigma,k}$, and $k$-NN. In the last two rows, the median misclassification rate of the outsiders for LDA and QDA, respectively.}
\vspace{0.5cm}  
 \centering
\small\addtolength{\tabcolsep}{-3pt}
\begin{tabular}{ p{2.2cm}  p{1.5cm}  p{1.5cm}  p{1.5cm} p{.6cm}  p{1.5cm}  p{1.5cm}  p{1.5cm}  p{.6cm} p{1.5cm} }
& \multicolumn{3}{c}{ Bivariate normal} & & \multicolumn{3}{c}{ Bivariate elliptical} & &Circular \\
& location & scale & location \& scale & &location & scale & location \& scale & &\\
\hline
$d_{\triangle_{\sigma},n}$ & {\bf{0.038}} & {\bf{0}} & {\bf{0}}& & 0.038 & {\bf{0}} & 0.103 & &{\bf{0}} \\
ref.\ halfspace & 0.191 & {\bf{0}} & {\bf{0}}& & 0.042 & {\bf{0}}&  0.107 && {\bf{0}} \\
illumination & {\bf{0.038}} & {\bf{0}} & {\bf{0}}& & 0.039 & {\bf{0}} & {\bf{0.098}} & &{\bf{0}} \\
$d_{P_\sigma,k}$ & 0.053 & {\bf{0}} & {\bf{0}}& & 0.045 & {\bf{0}}&  0.120 & &0.050 \\
$k$-NN & 0.043 & {\bf{0}} & {\bf{0}} & &0.069 & {\bf{0}} & 0.107 & &{\bf{0}} \\
LDA & {\bf{0.038}} & 0.551 & 0.423 && {\bf{0.029}} & 0.523 & 0.241 & &0.538 \\
QDA & {\bf{0.038}} & {\bf{0}} & {\bf{0}} && 0.032 & {\bf{0}} & 0.108 & &0.564 \\
\end{tabular}
\end{table}

\end{document}